\documentclass[11pt,reqno]{article}
\usepackage{amsmath}
\usepackage{mathrsfs}
\usepackage{amsfonts}
\usepackage{amssymb}

\usepackage{amssymb}
\usepackage{amsthm}
\usepackage{graphicx}              % to include figures
\usepackage{amsmath}               % great math stuff
\usepackage{amsfonts}              % for blackboard bold, etc
\usepackage{amsthm}                % better theorem environments
\usepackage{setspace}
\usepackage{epstopdf}
\usepackage{epsfig}

\newtheorem{theorem}{Theorem}[section]
\newtheorem{proposition}{Proposition}[section]
\newtheorem{lemma}{Lemma}[section]
\newtheorem{corollary}{Corollary}[section]
\newtheorem{remark}{Remark}[section]
\newtheorem{definition}{Definition}[section]

\renewcommand{\thefootnote}{\fnsymbol{footnote}}
\makeatletter

\newcommand{\Rmnum}[1]{\expandafter\@slowromancap\romannumeral #1@}
\makeatother

\allowdisplaybreaks
\numberwithin{equation}{section}

\title{%\Large\bf\boldmath
Convergence to harmonic maps for the Landau-Lifshitz flows on two dimensional hyperbolic spaces}

\author{ Ze Li\qquad Lifeng Zhao }

\date{}

\begin{document}

\maketitle

\renewcommand{\thefootnote}{\fnsymbol{footnote}}
\footnotetext{\hspace*{-5mm}
\begin{tabular}{@{}r@{}p{16cm}@{}}
Ze Li, E-mail: rikudosennin@163.com; \mbox{  }Lifeng Zhao, E-mail: zhaolf@ustc.edu.cn.
\end{tabular}}

\noindent{\bf Abstract}
In this paper, we prove that the solution of the Landau-Lifshitz flow $u(t,x)$ from $\mathbb{H}^2$ to $\mathbb{H}^2$ converges to some harmonic map as $t\to\infty$. The essential observation is that although there exist infinite numbers of harmonic maps from $\Bbb H^2$ to $\Bbb H^2$, the heat flow initiated from $u(t,x)$ for any given $t>0$ converges to the same harmonic map as the heat flow initiated from $u(0,x)$. This observation enables us to construct a variant of Tao's caloric gauge to reduce the convergence to harmonic maps for the Landau-Lifshitz flow to the decay of the corresponding heat tension field.  The advantage of the strategy used in this paper is that we can see the limit harmonic map directly by evolving $u(0,x)$ along a heat flow without evolving the Landau-Lifshitz flow to the infinite time.

%\noindent{\bf MR Subject Classification:} XXXx.

\bigskip

\section{Introduction}
Let $(M,h)$ be a Riemannian manifold and $(N,J,g)$ be a K\"ahler manifold, the Landau-Lifshitz flow is a map $u(x,t):M\times [0,\infty)\to N$ satisfying
\begin{equation}\label{1}
\begin{cases}
u_t =\alpha\tau(u)-\beta J(u)\tau(u)\\
u\upharpoonright_{t=0} = u_0(x),
\end{cases}
\end{equation}
where $\alpha\ge0$, $\beta\in\Bbb R$.
In the local coordinates $(x_1,x_2)$ for $M$ and $(y_1,y_2)$ for $N$, $\tau(u)$ is given by
$$
\tau(u)=\big(\Delta_{\mathbb{H}^2}u^{l}+h^{ij}\overline{\Gamma}^l_{m,n}(u)\frac{\partial u^{m}}{\partial x^i}\frac{\partial u^{n}}{\partial x^j}\big)\frac{\partial}{\partial y^l},
$$
where $h_{ij}dx^idx^j$ is the metric tension for $M$, $(h^{ij})$ is its inverse, $\overline{\Gamma}^{l}_{m,n}(u)$ is the Christoffel symbol at $u$. Usually, $\alpha\ge0$ is called the Gilbert constant. When $\alpha=0$, (\ref{1}) is called the Schr\"odinger flow. When $\beta=0$, $\alpha>0$, it reduces to the heat flows of harmonic maps.
In this paper, we consider the case when $M=\Bbb H^2$ and $N=\Bbb H^2$.

Besides the physical motivation, such as the continuous isotropic Heisenberg spin model, the gauge theory, the dynamics of the magnetization field inside ferromagnetic material (Landau-Lifshitz \cite{LL}), the Landau-Lifshitz flow (LL) is also a typical model in the differential geometry.  We recall the following non-exhaustive list of works. The heat flow (HF) case $(\alpha>0, \beta=0)$ has been intensively studied in the past 60 years, for instance Eells and Sampson \cite{ES} for HF from closed manifolds to closed manifolds, Hamilton \cite{Har} for HF with Dirichlet boundary condition, Li and Tam \cite{LT} for HF from complete manifolds to complete manifolds. When $\alpha>0,\beta\in\Bbb R$, the local well-posedness and partial regularity for weak solutions were considered by many authors for instance \cite{GH,M,ZGT}.  The local well-posedness of Schr\"odinger flow was studied by Sulem, Sulem and Bardos \cite{SSB} for $\mathbb{S}^2$ targets, Ding and Wang \cite{DW}, McGahagan \cite{Mc} for general K\"ahler manifolds. Chang, Shatah and Uhlenbeck \cite{CSU} obtained the global well-posedness of maps into closed Riemann surfaces for small initial data. The global well-posedness for maps from $\Bbb R^d$ into $\mathbb{S}^2$ with small critical Sobolev norms was proved by Bejenaru, Ionescu, Kenig and Tataru \cite{BIK, BIKT1}. The one dimensional case was studied by Rodnianski, Rubinstein and Staffilani \cite{RRS}.

The dynamic behavior of LL is known in some cases.  For the equivariant Schr\"odinger flows from $\Bbb R^2$ into $\Bbb S^2$ with energy below the ground state and equivariant flows from $\Bbb R^2$ into $\Bbb H^2$ with initial data of finite energy, the global well-posedness and scattering in the gauge sense were proved by Bejenaru, Ionescu, Kenig and Tataru \cite{BIKT2, BIKT3}. For the m-equivariant LL from $\Bbb R^2$ into $\Bbb S^2$ with initial data near the harmonic map, Gustafson, Kang, Tsai \cite{GKT1,GKT2} proved asymptotic stability for $m\ge4$ and later the case $m=3$ was proved by Gustafson,  Nakanishi, Tsai \cite{GNT}. Moreover, there exist blow up solutions near the harmonic maps, see Merle, Raphael, Rodnianski \cite{MPR} and Perelman \cite{P} for 1-equivariant 2D Schr\"odinger maps, and Chang, Ding, Ye \cite{CDY} for 2D symmetric heat flows.

Usually the dynamics for flows defined on the Euclidean space and curved space are typically different and of independent interest. For the heat flow, $u:[0,T]\times\Bbb H^n\to\Bbb H^n$ is of particular interest because it is closely related to the Schoen-Li-Wang conjecture, namely any quasiconformal boundary map gives rise to a harmonic map of hyperbolic space. (see for instance Lemm, Markovic \cite{LM}) The wave map dynamics on curved space especially hyperbolic space were studied in the sequel works of Lawrie, Oh, Shahshahani \cite{LA},  \cite{LSO1,LSO2,LSO}.

In this paper, we study the long time behaviors of solutions to (\ref{1}) for $\alpha>0$.  In our previous paper \cite{LZ}, for LL from $\Bbb R^2$ to a compact Riemann surface, we proved that any initial data with energy below the critical energy will evolve to a global solution and converge to a constant map in the energy space. For LL from $\mathbb{H}^2$ to $\mathbb{H}^2$, we aim to prove that the solution exists globally and will converge to a harmonic map as $t\to\infty$.  In order to prove the convergence to harmonic maps, one may study the corresponding linearized equation, which works well for the equivariant case. However, the linearization method is not available in the general case because there are infinite numbers of harmonic maps. In this paper, we apply the caloric gauge technique initially introduced by Tao \cite{Tao}.

The caloric gauge originally used by Tao was applied to solve the global regularity of wave maps from $\Bbb R^{2+1}$ to $\Bbb H^n$ in a sequel of papers \cite{Tao3,Tao4}. We briefly recall the main idea of caloric gauge in the wave map setting. The first step is to evolve the solution of the wave map $u(t,x)$ along a heat flow, i.e., solve the heat flow equation for $\widetilde{u}(s,t,x)$ with initial data $u(t,x)$. If there exists no non-trivial harmonic map,  one can suppose that the corresponding heat flow converges to a fixed point $Q$. For any given orthonormal frame at the point $Q$, we can pullback the orthonormal frame parallel with respect to $s$ along the heat flow to obtain the frame at $\widetilde{u}(s,t,x)$, especially $u(t,x)$ when $s=0$. Then one has a scalar system for the differential fields and connection coefficients after rewriting (\ref{1}) under the constructed frame. The caloric gauge can be seen as a nonlinear Littlewood-Paley decomposition, and it removes some troublesome frequency interactions.
Generally the caloric gauge works well below the critical energy, where no harmonic map occurs.  In our case, it makes no sense to study the dynamics below the critical energy because for any given $\lambda\in (0,\infty)$, there exists a harmonic map $Q_{\lambda}$ whose energy is exactly $\lambda$.  However, there is an obvious advantage of the caloric gauge in our case, i.e., one can a priori see the limit harmonic map without evolving LL to the infinity time. In fact, denote the solution of the heat flow with initial data $u(0,x)$ by $U(t,x)$, then it is known that $U(t,x)$ converges to some harmonic map $Q(x)$ as $t\to\infty$. The key observation is that one can still expect that the solution $u(t,x)$ of (\ref{1}) also converges to the same harmonic map $Q(x)$ as $t\to\infty$. This informal heuristic idea combined with the caloric gauge reduces the convergence of LL to proving the decay of the heat tension filed.  There are two main obstacles while applying the caloric gauge in the appearance of non-trivial harmonic maps. One is to guarantee that all the heat flows initiated from $u(t,x)$ for different $t$ converge to the same harmonic map. The other is the lack of global bounds for the derivatives of the induced connection coefficient. In fact, since formally the induced connection coefficient is solved from the infinity time, an $L^1$ integrability with respect to time is needed, but harmonic maps prevent the energy from decaying to zero as time goes to infinity, then we have no enough decay to gain the integrability.

The first obstacle is overcome by using the structure of LL. In order to construct the caloric gauge, one needs to prove the heat flow with $u(t,x)$ as the initial data converges to the same harmonic map independent of $t$. We remark that it is not a trivial fact, because it is false if we only consider $t$ as a smooth parameter, i.e., in the homotopy class. Indeed, it is known that there exist a family of harmonic maps $\{Q_{\lambda}\}$ which depend smoothly with respect to $\lambda\in(0,1)$. Then perturbing the initial data $Q_{\lambda}$ of the heat flow to be $Q_{\lambda'}$ yields a different limit harmonic map, since any harmonic map remains time-independent under the heat flow. This inspires us to use the structure of LL.  The key point is to study the evolution of $\partial_t u$ with respect to the heat flow. By a monotonous property observed first by Hartman \cite{Hh} and the decay estimates of the heat semigroup, we can prove the distance between the heat flows initiated from $u(t_1)$ and $u(t_2)$ goes to zero as $s\to\infty$. Thus the limit harmonic map for the heat flow generated from $u(x,t)$ are all the same for different $t$.

We remove the second obstacle by the smoothing effect. After constructing the caloric gauge, we obtain the gauged equation for the corresponding differential fields $\phi_i$ and connection coefficients $A_i$. Then it suffices to prove the decay of the corresponding heat tension field governed by a Ginzburg-Landau type system (see (\ref{heat})). The most difficult term in this system is $\nabla A_i$, the derivative of the connection coefficients.  Because the energy of the solution to (\ref{1}) is strictly away from zero due to the appearance of the harmonic map, one can not expect any integrability of $\phi_i$ with respect to $s$. Thus if one tries to obtain a global bound for $\nabla A_i$, one has to put $\nabla\partial_s\widetilde{u}$ in $L^1_s$. (see (\ref{edf}) for the expression for $A_i$) The term $\nabla\partial_s\widetilde{u}$ might be put in $L^1_s$ by constructing a proper auxiliary function which satisfies a proper semiliear heat equation and controls $\nabla\partial_s\widetilde{u}$ and applying the maximum principle to obtain a pointwise estimate. Instead of constructing such an auxiliary function, we use an indirect but more robust method. By applying the smoothing effect of the heat semigroup, it suffices to bound $\nabla A_i\phi_s$ in $\dot{H}^{-1}$. Then by duality, one can move the derivative from $\partial_iA_i$ to $\phi_s$. Since we have a a-prior bound for $\nabla\phi_s$ in terms of $\nabla\partial_tu$, which is also a-prior bounded in $L^{2}_{t,x}$ through some delicate energy arguments for (\ref{1}), the $\partial_iA_i\phi_s$ term can be tackled without studying $\partial_iA_i$ itself.

In order to state our main theorem, we introduce some notions.
\begin{definition}
We say $Q:\mathbb{H}^2 \to \mathbb{H}^2$ is an admissible harmonic map if $\overline{Q(\Bbb H^2)}$ is compact and $\nabla^kQ\in L^2(\Bbb H^2)$ for $k\in\{1,2,3\}$. For this $Q$, the space $\mathcal{H}_{Q}^3$ is defined by (\ref{h897}).  Denote the union of $\mathcal{H}_{Q}^3$ with $Q$ ranging over all the admissible harmonic maps by $\mathrm{H}^3$.
\end{definition}

{\bf Remark 1.1}
We remark that any map which coincides with $Q$ outside of a compact subset of $\Bbb H^2$ is a member of $\mathcal{H}_{Q}^3$. Any analytic function $f:\Bbb C\to\Bbb C$ with $f(\Bbb D)\Subset\Bbb D$ where $\Bbb D$ is the Poincare disk is an admissible harmonic map, see Appendix for the proof. The harmonic map studied in Lawrie, Oh, Shahshahani \cite{LSO7} is in fact $f(z)=\lambda z$ with $\lambda\in(0,1)$.

The main result of this paper is the following.
\begin{theorem}\label{aaa.1}
Let $\alpha>0$, $\beta\in \Bbb R$. For any initial data $u_0\in \mathrm{H}^3$, there exists a global solution to (\ref{1}) and as $t\to\infty$, $u(t,x)$ converges to some harmonic map $Q_{\infty}:\Bbb H^2\to\Bbb H^2$, namely
$$
\mathop {\lim }\limits_{t \to \infty } \mathop {\sup }\limits_{x \in {\mathbb{H}^2}} {d_{{\mathbb{H}^2}}}\left( {u(t,x),Q_{\infty}(x)} \right) = 0.
$$
\end{theorem}

{\bf Remark 1.2} The limit harmonic map $Q_{\infty}$ in Theorem 1.1 is in fact $Q$ if $u_0\in \mathcal{H}^3_Q$. This is related to the uniqueness of harmonic maps with prescribed boundary harmonic map. But we will not go into further details on this problem in this paper, one may see [Lemma 2.4, \cite{Lize}] for details.

This paper is organized as follows. In Section 2, we recall some background materials and prove an equivalence lemma for the intrinsic and extrinsic Sobolev norms in some case. In Section 3, we construct the caloric gauge and obtain the estimates of the connection coefficients. In Section 4, we prove the local and global well-posedness of (\ref{1}) by energy arguments. In Section 5, we study the gauged system and prove the decay of the heat tension filed which concludes the proof of Theorem \ref{aaa.1}.

\section{\bf Preliminaries}
In this section, we recall some standard preliminaries on the geometry notions of the hyperbolic spaces along with some Sobolev embedding inequalities. For the study of the well-posedness theories, we need to use both the intrinsic and extrinsic formulations of the Sobolev spaces, thus we establish an equivalence relationship for the two formulations in some case, namely Lemma \ref{new}.

\subsection{The global coordinates and definitions of the function spaces}
The covariant derivative in $TN$ is denoted by $\widetilde{\nabla}$, the covariant derivative induced by $u$ in $u^*(TN)$ is denoted by $\nabla$. The Riemann curvature tension of $N$ is denoted by $\mathbf{R}$. The components of Riemann metric are denoted by $h_{ij}$ for $M$ and $g_{ij}$ for $N$. And $\Gamma^{k}_{lj}$, $\overline{\Gamma}^{k}_{lj}$ denote the Christoffel symbols for $M$ and $N$ respectively.

We recall some facts on hyperbolic spaces. Let $\Bbb R^{2+1}$ be the Minkowski space with the Minkowski metric $-(dx^0)^2+(dx^1)^2+(dx^2)^2$. Define a bilinear form on $\Bbb R^{2+1}\times \Bbb R^{2+1}$,
$$
[x,y]=x^0y^0-x^1y^1-x^2y^2.
$$
The hyperbolic space $\mathbb{H}^2$ is defined as
$$\mathbb{H}^2=\{x\in \Bbb R^{2+1}: [x,x]=1 \mbox{  }{\rm{and}}\mbox{  }x^0>0\}.$$
The Riemann metric equipped on $\mathbb{H}^2$ is the pullback of the Minkowski metric by the inclusion map $\iota:\mathbb{H}^2\to \Bbb R^{2+1}.$
The Iwasawa decomposition gives a global system of coordinates. Define the diffeomorphism $\Psi:\Bbb R\times \Bbb R\to \mathbb{H}^2$,
\begin{align}\label{vg}
\Psi(x_1,x_2)=({\rm{cosh}} x_2+e^{-x_2}|x_1|^2/2, {\rm{sinh}} x_2+e^{-x_2}|x_1|^2/2, e^{-x_2}x_1).
\end{align}
The Riemann metric with respect to this coordinate system is given by
$$
e^{-2x_2}(dx_1)^2+(dx_2)^2.
$$
The corresponding Christoffel symbols are
\begin{align}\label{christ}
\Gamma^1_{2,2}=\Gamma^2_{2,1}=\Gamma^2_{2,2}=\Gamma^1_{1,1}=0; \mbox{ }\Gamma^1_{1,2}=-1,, \mbox{  }\Gamma^2_{1,1}=e^{-2x_2}.
\end{align}
And $J(\frac{\partial}{\partial x_1})=e^{-x_2}\frac{\partial}{\partial x_2}$. For any $(t,x)$ and $u:[0,T]\times\mathbb{H}^2\to \Bbb H^2$, we define an orthonormal frame at $u(t,x)$ by
\begin{align}\label{frame}
\Theta_1(u(t,x))=e^{u^2(t,x)}\frac{\partial}{\partial y_1}; \mbox{  }\Theta_2(u(t,x))=\frac{\partial}{\partial y_2}.
\end{align}
It is easily seen $\Theta_2=J\Theta_1$.
Let $X,Y,Z\in TN$, recall the identity for Riemannian curvature on $\Bbb H^d$
\begin{align*}
{\widetilde{\nabla} _X}{\widetilde{\nabla} _Y}Z - {\widetilde{\nabla }_Y}{\widetilde{\nabla }_X}Z - {\widetilde{\nabla }_{[X,Y]}}Z = \left( {X \wedge Y} \right)Z,
\end{align*}
where we use the simplicity notation
\begin{align*}
\left( {X \wedge Y} \right)Z = \left\langle {X,Z} \right\rangle Y - \left\langle {Y,Z} \right\rangle X.
\end{align*}
As a direct consequence of this formula and the comparability, one has for $X,Y,Z\in u^{*}TN$ that
\begin{align*}
&{\nabla _i }\left( {\bf R\left( {X,Y} \right)Z} \right) = {\nabla _i }\left( {\left\langle {X,Z} \right\rangle Y} \right) - {\nabla _i }\left( {\left\langle {Y,Z} \right\rangle X} \right) \\
&= \left\langle {X,Z} \right\rangle {\nabla _i }Y + \left\langle {{\nabla _i }X,Z} \right\rangle Y + \left\langle {X,{\nabla _i }Z} \right\rangle Y - \left\langle {Y,Z} \right\rangle {\nabla _i }X - \left\langle {{\nabla _i }Y,Z} \right\rangle X - \left\langle {Y,{\nabla _i }Z} \right\rangle X \\
&= {\bf R}\left( {X,{\nabla _i}Y} \right)Z + {\bf R}\left( {{\nabla _i }X,Y} \right)Z + {\bf R}\left( {X,Y} \right){\nabla _i }Z.
\end{align*}
Therefore we obtain a useful identity
\begin{align}\label{curvature2}
{\nabla _i }\left( {{\bf R}\left( {X,Y} \right)Z} \right) = {\bf R}\left( {X,{\nabla _i }Y} \right)Z + {\bf R}\left( {{\nabla _i }X,Y} \right)Z + {\bf R}\left( {X,Y} \right){\nabla _i }Z.
\end{align}

Let $H^k(\mathbb{H}^2;\Bbb R)$ be the usual Sobolev space for scalar functions defined on manifolds, see for instance Hebey \cite{Hebey}. It is known that $C^{\infty}_c(\mathbb{H}^2;\Bbb R)$ is dense in $H^k(\mathbb{H}^2;\Bbb R)$. We also recall the norm of $H^k$:
$$
\|f\|^2_{H^k}=\sum^k_{l=1}\|\nabla^l f\|^2_{L^2_x},
$$
where $\nabla^l f$ is the covariant derivative, for instance, $\nabla f$ and $\nabla^2 f$ in the local coordinates can be written as
$$
\nabla f=\frac{\partial f}{\partial x_i}dx^i, \mbox{  }\nabla^2 f=\left( {\frac{{{\partial ^2}f}}{{\partial {x_i}\partial {x_j}}} - \Gamma _{ij}^k\frac{{\partial f}}{{\partial {x_k}}}} \right)d{x^i} \otimes d{x^j}.
$$
The norm $|\nabla^l f|$ is taken by viewing $\nabla^l f$ as a $(0,l)$ type tension field on $\mathbb{H}^2$.
For maps $u:\mathbb{H}^2\to \mathbb{H}^2$, the intrinsic Sobolev semi-norm $\mathfrak{H}^k$ is given by
$$
\|u\|^2_{{\mathfrak{H}}^k}=\sum^k_{i=1}\int_{\mathbb{H}^2} |\nabla^{i-1} du|^2 {\rm{dvol_h}}.
$$
Recall the global coordinates given by (\ref{vg}), then map $u:\mathbb{H}^2\to\mathbb{H}^2$ can be viewed as a vector-valued function $u:\mathbb{H}^2\to \Bbb R^2$. Indeed, for $P\in \mathbb{H}^2$ the vector $(u^1(P),u^2(P))$ is defined by $\Psi(u^1(P),u^2(P))=u(P)$. Let $Q:\Bbb H^2\to\Bbb H^2$ be an admissible harmonic map. Then the extrinsic Sobolev space is defined as
\begin{align}\label{h897}
\mathcal{H}^k_{Q}=\{u: u^1-Q^1(x), u^2-Q^2(x)\in H^k(\mathbb{H}^2;\Bbb R)\},
\end{align}
where $(Q^1(x),Q^2(x))\in \Bbb R^2$ is the corresponding components of $Q(x)$ under the coordinate given by (\ref{vg}). Denote the union of $\mathcal{H}_{Q}^3$ with $Q$ ranging over all the admissible harmonic maps by $\mathrm{H}^3$. Equip $\mathcal{H}^k_{Q}$ with the following metric
\begin{align}\label{poi65gv}
{\rm{dist}}_{Q,k}(u,w)=\sum^{2}_{j=1}\|u^j-w^j\|_{H^k},
\end{align}
for $u,w\in\mathcal{H}^k_{Q}$. Since $C^{\infty}_c(\Bbb H^2,\Bbb R^2)$ is dense in $H^k$ (see Hebey \cite{Hebey}), $\mathcal{H}^k_{Q}$ coincides with the completion of $\mathcal{D}$, which denotes the maps from $\Bbb H^2$ to $\Bbb H^2$ coinciding with $Q$ outside of some compact subset of $M=\Bbb H^2$, under the metric given by
(\ref{poi65gv}). In the following, without of confusing we write $\mathcal{H}^k_Q$ as $\mathcal{H}^k$ for simplicity.

\subsection{The Fourier transform on hyperbolic spaces and Sobolev embedding}
The Fourier transform takes proper functions defined on $\mathbb{H}^2$ to functions defined on $\Bbb R\times S^1$, see for instance Helgason \cite{Hel}.
For $\omega\in S^1$, and $\lambda\in \Bbb C$, let $b(\omega)=(1,\omega)\in \Bbb R^3$, we define
$$
h_{\lambda,\omega}:\mathbb{H}^2\to \Bbb C, \mbox{  }h_{\lambda,\omega}=[x,b(\omega)]^{i\lambda-\frac{1}{2}}.
$$
The Fourier transform of $f\in C_0(\mathbb{H}^2)$ is defined by
$$
\widetilde{f}\left( {\lambda ,\omega } \right) = \int_{{\Bbb H^2}} {f(x){h_{\lambda ,\omega }}(x){\rm{dvol_h}} = } \int_{{\Bbb H^2}} {f(x){{\left[ {x,b(\omega )} \right]}^{i\lambda  - \frac{1}{2}}}{\rm{dvol_h}}}.
$$
The corresponding Fourier inversion formula is given by
$$
f(x) = \int_0^\infty  {\int_{{S^1}} {f\left( {\lambda ,\omega } \right)} } {\left[ {x,b(\omega )} \right]^{ - i\lambda  - \frac{1}{2}}}{\left| {c(\lambda )} \right|^{ - 2}}d\lambda d\omega,
$$
where $c(\lambda)$ is the Harish-Chandra c-function on $\mathbb{H}^2$, which is defined for some suitable constant $C$ by
$$
c(\lambda)=C\frac{\Gamma(i\lambda)}{\Gamma(\frac{1}{2}+i\lambda)}.
$$
The Plancherel theorem is as follows
$$
\int_{{\mathbb{H}^2}} {f(x)\overline {g(x)}{\rm{dvol_h}} = \frac{1}{2}} \int_{\mathbb{R} \times {S^1}} {f\left( {\lambda ,\omega } \right)\overline {g\left( {\lambda ,\omega } \right)} {{\left| {c(\lambda )} \right|}^{ - 2}}d\lambda d\omega }.
$$
Thus any bounded multiplier $m:\Bbb R\to \Bbb C$ defines a bounded operator $T_m$ on $L^2(\mathbb{H}^2)$ by
$$
\widetilde{T_m(f)}(\lambda, \omega)=m(\lambda)\widetilde{f}(\lambda,\omega).
$$
We define the operator $(-\Delta)^{\frac{s}{2}}$ by the Fourier multiplier $\lambda\to (\frac{1}{4}+\lambda^2)^{\frac{s}{2}}$.
We now recall the Sobolev inequalities of functions in $H^k$.

\begin{lemma}\label{wusijue}
If $f\in H^{3}(\mathbb{H}^2)$, then for $1<p<\infty,$ $p\le q\le \infty$, $0<\theta<1$, $1<r<2$, $r\le l<\infty$, $1<\alpha\le 3$ following inequalities hold
\begin{align}
{\left\| f \right\|_{{L^2}}} &\lesssim {\left\| {\nabla f} \right\|_{{L^2}}} \label{uv1}\\
 {\left\| f \right\|_{{L^q}}} &\lesssim \left\| {\nabla f} \right\|_{{L^2}}^\theta \left\| f \right\|_{{L^p}}^{1 - \theta }\mbox{  }{\rm{when}}\mbox{  }\frac{1}{p} - \frac{\theta }{2} = \frac{1}{q} \label{uv2}\\
 {\left\| f \right\|_{{L^l}}} &\lesssim {\left\| {\nabla f} \right\|_{{L^r}}}\mbox{  }\mbox{  }\mbox{  }\mbox{  }\mbox{  }\mbox{  }\mbox{  }\mbox{  }\mbox{  }{\rm{when}}\mbox{  }\frac{1}{r} - \frac{1}{2} = \frac{1}{l} \label{uv3}\\
 {\left\| f \right\|_{{L^\infty }}} &\lesssim {\left\| {{{\left( { - \Delta } \right)}^{\frac{\alpha }{2}}}f} \right\|_{{L^2}}}\mbox{  }\mbox{  }\mbox{  }{\rm{when}}\mbox{  }\alpha>1 \label{uv4}\\
 {\left\| {\nabla f} \right\|_{{L^2}}} &\sim{\left\| {{{\left( { - \Delta } \right)}^{\frac{1}{2}}}f} \right\|_{{L^2}}} \label{uv5}.
 \end{align}
\end{lemma}
For the proof, we refer to Bray \cite{Bray} for (\ref{uv2}), Ionescu, Pausader, Staffilani \cite{IPS} for (\ref{uv3}), Hebey \cite{Hebey} for (\ref{uv4}), see also Lawrie, Oh, Shahshahani \cite{LSO}.

We also recall the standard diamagnetic inequality which sometimes refers to Kato's inequality as well.
\begin{lemma}\label{wusijue3}
If $T$ is some $(r,s)$ type tension or tension matrix defined on $\Bbb H^2$, then in the distribution sense, one has the diamagnetic inequality
$$
|\nabla|T||\le |\nabla T|.
$$
\end{lemma}

\begin{remark}
Combining Lemma \ref{wusijue} and Lemma \ref{wusijue3}, we have several useful corollaries, for instance
\begin{align}\label{uv16}
\|f\|_{L_x^{\infty}}\le \|\nabla^2f\|_{L^2_x}, \mbox{  }{\left\| f \right\|_{{{\dot H}^{ - 1}_x}}} \le {\left\| f \right\|_{{L^2_x}}}.
\end{align}
\end{remark}

\subsection{Equivalence Lemma}
In this subsection, we prove the equivalence of the intrinsic Sobolev space and the extrinsic Sobolev space in some case.
Suppose that $Q$ is an admissible harmonic map in Definition 1.1.
Although the geodesic distance between $(0,0)$ and $(x_1,x_2)$ in the coordinate (\ref{vg}) is not $|x_1|+|x_2|$, we use the quantity $|x_1|+|x_2|$ in the following three lemmas for simplicity. It suffices to remind ourself that the true distance is bounded by a function of $(|x_1|,|x_2|)$.

As a preparation, we give the following lemma which shows the intrinsic formulation is equivalent to the extrinsic one if the image of the map is compact. \begin{lemma}\label{hbvcm7}
If $u:\Bbb H^2\to\Bbb H^2$ with $u(\Bbb H^2)$ covered by a geodesic ball of radius $R$, then for $d=1,2$
\begin{align}
 {\| {\nabla {u^d}}\|_{{L^2}}} &\le C(R){\| {du} \|_{{L^2}}} \label{miao1}\\
 {\| {{\nabla ^2}{u^d}} \|_{{L^2}}}&\le C(R){\| {\nabla du}\|_{{L^2}}} + C(R)\| {\nabla du}\|_{{L^2}}^2 \label{miao2}\\
 {\| {du}\|_{{L^2}}} &\le C(R){\| {\nabla {u^1}}\|_{{L^2}}} + C(R){\| {\nabla {u^2}}\|_{{L^2}}} \label{miao3}\\
 {\| {\nabla du}\|_{{L^2}}} &\le \sum\limits_{k = 1}^2 {C(R)} {\| {{\nabla^2}{u^k}}\|_{{L^2}}} + C(R){\| {{\nabla^2}{u^k}} \|^2_{{L^2}}}.\label{miao4}
\end{align}
\end{lemma}
\begin{proof}
With the coordinates (\ref{vg}), we have
\begin{align}
 {\left| {du} \right|^2} &= {h^{ii}}{\left| {{\partial _{{x_i}}}{u^j}} \right|^2}{g _{jj}} = {h^{ii}}{\left| {{\partial _{{x_i}}}{u^1}} \right|^2}{e^{ - 2{u^2}}} + {h^{ii}}{\left| {{\partial _{{x_i}}}{u^2}} \right|^2} \label{miao6}\\
 {\left| {\nabla du} \right|^2} &= {\left( {\frac{{{\partial ^2}{u^l}}}{{\partial {x_j}\partial {x_i}}} + {\partial _j}{u^k}{\partial _i}{u^m}\overline \Gamma  _{mk}^l - \Gamma _{ij}^k{\partial _k}{u^l}} \right)^2}{g_{ll}}{h^{ii}}{h^{jj}} \label{miao7}\\
 {\left| {\nabla {u^j}} \right|^2} &= {\left| {{\partial _{{x_i}}}u{}^j} \right|^2}{h^{ii}} \label{miao8}\\
 {\left| {{\nabla ^2}{u^l}} \right|^2} &= {\left| {\frac{{{\partial ^2}{u^l}}}{{\partial {x_j}\partial {x_i}}} - \Gamma _{ij}^k{\partial _k}{u^l}} \right|^2}{h^{ii}}{h^{jj}}.\label{miao9}
\end{align}
Thus (\ref{miao1}) and (\ref{miao3}) directly follow by (\ref{miao8}) and (\ref{miao6}) with $C(R)=e^{2R}$.
Using the explicit formula for $\Gamma^{k}_{ij}$, (\ref{miao7}) and (\ref{miao9}) give for $C(R)=e^{8R}$,
\begin{align*}
{\left| {{\nabla ^2}{u^d}} \right|^2} &\le {\left| {\nabla du} \right|^2} + {\left| {{\partial _j}{u^k}{\partial _i}{u^m}\overline \Gamma  _{mk}^l} \right|^2}{g_{ll}}{h^{ii}}{h^{jj}} \\
&\lesssim {\left| {\nabla du} \right|^2} + C(R){\left| {e(u)} \right|^2}.
\end{align*}
Then Sobolev embedding (\ref{uv3}) and diamagnetic inequality show
\begin{align*}
\int_{{\mathbb{H}^2}} {{{\left| {{\nabla ^2}{u^d}} \right|}^2}}{\rm{dvol_h}}&\le C(R)\int_{{\mathbb{H}^2}} {{{\left| {\nabla du} \right|}^2}}{\rm{dvol_h}} + C(R)\int_{{\Bbb H^2}} {{{\left| {e(u)} \right|}^2}}{\rm{dvol_h}} \\
&\le C(R)\int_{{\mathbb{H}^2}} {{{\left| {\nabla du} \right|}^2}}{\rm{dvol_h}} +C(R) {\left[ {\int_{{\Bbb H^2}} {{{\left| {\nabla \sqrt {e(u)} } \right|}^2}}{\rm{dvol_h}}} \right]^2} \\
&\le C(R) \int_{{\mathbb{H}^2}} {{{\left| {\nabla du} \right|}^2}}{\rm{dvol_h}}+C(R){\left[\int_{{\mathbb{H}^2}} {{{\left| {\nabla du} \right|}^2}}{\rm{dvol_h}}\right]}^2.
\end{align*}
Hence (\ref{miao2}) is obtained.
Similarly we have for $C(R)=e^{8R}$
\begin{align*}
 {\left| {\nabla du} \right|^2} &\le {g_{ll}}{\left| {{\nabla ^2}{u^l}} \right|^2} + {\left| {{\partial _j}{u^k}{\partial _i}{u^m}\bar \Gamma _{mk}^l} \right|^2}{g_{ll}}{h^{ii}}{h^{jj}} \\
 &\le C(R){\left| {{\nabla ^2}{u^l}} \right|^2} + C(R){\left| {\nabla {u^k}} \right|^4}.
 \end{align*}
Using Sobolev embedding (\ref{uv3}) again yields (\ref{miao4}).
\end{proof}

\begin{lemma}\label{new}
Let $Q$ be an admissible harmonic map with $Q(\Bbb H^2)$ contained in a geodesic ball of $N=\Bbb H^2$ with radius $R_0$. If $u\in \mathcal{H}^3$ then when $k=2,3$
\begin{align}
\|u\|_{\mathcal{H}^k}\thicksim\|u\|_{\mathfrak{H}^k},
\end{align}
in the sense that there exist continuous functions $\mathcal{Q}$ and $\mathcal{P}$ such that
\begin{align}
\|u\|_{\mathcal{H}^k}&\le C(\|u\|_{\mathfrak{H}^2},\|\nabla dQ\|_{L^2},R_0)\mathcal{P}(\|u\|_{\mathfrak{H}^k})\label{jia81}\\
\|u\|_{\mathfrak{H}^k}&\le C(\|u\|_{\mathcal{H}^2},R_0)\mathcal{Q}(\|u\|_{\mathcal{H}^k})\label{jia91}.
\end{align}
\end{lemma}
\begin{proof}
We first prove (\ref{jia81}).
By (\ref{miao6})-(\ref{miao9}),
\begin{align*}
 {\left| {{\nabla ^2}{u^2}} \right|^2} &\le {\left| {\nabla du} \right|^2} + {\left| {{\partial _j}{u^k}{\partial _i}{u^m}\overline \Gamma  _{mk}^l} \right|^2}{g_{ll}}{h^{ii}}{h^{jj}} \\
 &= {\left| {\nabla du} \right|^2} + 2{\left| {{\partial _j}{u^1}{\partial _i}{u^2}\overline \Gamma  _{21}^1} \right|^2}{e^{ - 2{u^2}}}{h^{ii}}{h^{jj}} + {\left| {{\partial _j}{u^1}{\partial _i}{u^1}\overline \Gamma  _{11}^2} \right|^2}{h^{ii}}{h^{jj}} \\
 &\lesssim  {\left| {\nabla du} \right|^2} + 2{\left| {{\partial _j}{u^1}{\partial _i}{u^2}} \right|^2}{e^{ - 2{u^2}}}{h^{ii}}{h^{jj}} + {\left| {{\partial _j}{u^1}{\partial _i}{u^1}} \right|^2}{e^{ - 4{u^2}}}{h^{ii}}{h^{jj}} \\
 &\lesssim {\left| {\nabla du} \right|^2} + {\left| {e(u)} \right|^2}.
\end{align*}
Then Sobolev embedding (\ref{uv3}) and diamagnetic inequality show
\begin{align}
\int_{{\mathbb{H}^2}} {{{\left| {{\nabla ^2}{u^2}} \right|}^2}}{\rm{dvol_h}}&\lesssim \int_{{\mathbb{H}^2}} {{{\left| {\nabla du} \right|}^2}}{\rm{dvol_h}} + \int_{{\Bbb H^2}} {{{\left| {e(u)} \right|}^2}}{\rm{dvol_h}} \nonumber\\
&\lesssim \int_{{\mathbb{H}^2}} {{{\left| {\nabla du} \right|}^2}}{\rm{dvol_h}} + {\left[ {\int_{{\Bbb H^2}} {{{\left| {\nabla \sqrt {e(u)} } \right|}^2}}{\rm{dvol_h}}} \right]^2} \nonumber\\
&\lesssim \int_{{\mathbb{H}^2}} {{{\left| {\nabla du} \right|}^2}}{\rm{dvol_h}}+{\left[\int_{{\mathbb{H}^2}} {{{\left| {\nabla du} \right|}^2}}{\rm{dvol_h}}\right]}^2.\label{hundao}
\end{align}
Recall the definition in (\ref{h897}), then (\ref{hundao}), (\ref{uv16}) with $Q\in \mathfrak{H}^2$ yield
\begin{align*}
&\|\nabla^2(u^2-Q^2)\|_{L^2}\le\|\nabla^2u^2\|_{L^2}+\|\nabla^2Q^2\|_{L^2}\\
&\lesssim\|\nabla dQ\|_{L^2}+\|\nabla dQ\|^2_{L^2}+\|\nabla du\|_{L^2}+\|\nabla du\|^2_{L^2}.
\end{align*}
Then by Sobolev embedding,
$$\|u^2-Q^2\|_{L^{\infty}}\lesssim\|\nabla^2(u^2-Q^2)\|_{L^2}\lesssim \sum^2_{k=1}\|\nabla dQ\|^k_{L^2}+\|\nabla du\|^k_{L^2}.
$$
Thus $\overline{\Gamma}^j_{ik}$ and $g_{ik}$ are bounded by the compactness of $Q(\Bbb H^2)$ and (\ref{christ}). Applying similar arguments to $u^1$ with the boundedness of $\overline{\Gamma}^j_{ik}$ and $g_{ij}$ implies
\begin{align}\label{hudfrtyu}
\|\nabla^2(u^1-Q^1)\|_{L^2}\lesssim\sum^2_{k=1}\|\nabla dQ\|^k_{L^2}+\|\nabla du\|^k_{L^2}.
\end{align}
where the implicit constant has an up bound $e^{8\big(\sum^2_{k=1}\|\nabla dQ\|^k_{L^2}+\|\nabla du\|^k_{L^2}+R_0\big)}$.
Thus the $k=2$ case in (\ref{jia81}) is proved. \\
By (\ref{hudfrtyu}) and Sobolev embedding, one has $\|u^1-Q^1\|_{\infty}<\infty$. Then the compactness of $Q(\Bbb H^2)$ indicates $u(\Bbb H^2)$ is covered by a geodesic ball of radius $CR_0+C\sum^2_{k=1}\|\nabla dQ\|^k_{L^2}+\|\nabla du\|^k_{L^2}$, where $C$ is some universal constant.
Therefore careful calculations give
\begin{align}
 {h^{ii}}{h^{jj}}{\left( {\Gamma _{ij}^k{\partial _k}{u^l}} \right)^2} &\lesssim {\left| {du} \right|^2} \label{1.1}\\
 {h^{ii}}{h^{jj}}{\left( {{\partial _j}{u^k}{\partial _i}{u^m}\overline \Gamma  _{mk}^l} \right)^2} &\lesssim{\left| {du} \right|^4} \label{1.2}\\
 {\left( {\frac{{{\partial ^2}{u^l}}}{{\partial {x_j}\partial {x_i}}}} \right)^2}{h^{ii}}{h^{jj}} &\lesssim{\left| {du} \right|^2} + {\left| {\nabla du} \right|^2} + {\left| {du} \right|^4}.\label{1.3}
\end{align}
The third order derivatives can be written as
\begin{align*}
 &{\left| {{\nabla ^3}{u^l}} \right|^2}\nonumber\\
 &= {h^{ii}}{h^{jj}}{h^{kk}}{\left| {\frac{{{\partial ^3}{u^l}}}{{\partial {x_j}\partial {x_i}\partial {x_k}}} - {\partial _k}\left( {\Gamma _{ij}^m{\partial _m}{u^l}} \right) - \Gamma _{ik}^m\left( {\frac{{{\partial ^2}{u^l}}}{{\partial {x_m}\partial {x_j}}} - \Gamma _{mj}^p{\partial _p}{u^l}} \right) - \Gamma _{jk}^m\left( {\frac{{{\partial ^2}{u^l}}}{{\partial {x_m}\partial {x_i}}} - \Gamma _{mi}^p{\partial _p}{u^l}} \right)} \right|^2},
 \end{align*}
 and
 \begin{align*}
 &{\left| {{\nabla ^2}du} \right|^2} \\
 &= {h^{ii}}{h^{jj}}{h^{kk}}{g_{ll}}\left\{ {\frac{{{\partial ^3}{u^l}}}{{\partial {x_j}\partial {x_i}\partial {x_k}}}} \right. - {\partial _k}\left( {\Gamma _{ij}^m(u){\partial _m}{u^l}} \right) - \Gamma _{ki}^p\left( {\frac{{{\partial ^2}{u^l}}}{{\partial {x_j}\partial {x_p}}} - \Gamma _{pj}^q{\partial _q}{u^l}} \right) - \Gamma _{k,j}^p\left( {\frac{{{\partial ^2}{u^l}}}{{\partial {x_i}\partial {x_p}}} - \Gamma _{pi}^q{\partial _q}{u^l}} \right) \\
 &+ {\partial _k}\left( {{\partial _j}{u^q}{\partial _i}{u^m}\overline \Gamma  _{mq}^l} \right) + \Gamma _{kq}^l\left( {\frac{{{\partial ^2}{u^q}}}{{\partial {x_j}\partial {x_i}}} + \left( {{\partial _j}{u^k}{\partial _i}{u^m}\overline \Gamma  _{mk}^q - \Gamma _{ij}^m(u){\partial _m}{u^q}} \right)} \right) \\
 &- \Gamma _{ki}^p{\partial _j}{u^q}{\partial _p}{u^m}\overline \Gamma  _{mq}^l\left. { - \Gamma _{k,j}^p{\partial _i}{u^q}{\partial _p}{u^m}\overline \Gamma  _{mq}^l} \right\}^2.
 \end{align*}
Then it suffices to bound
\begin{align}
&\int_{{\mathbb{H}^2}} {{h^{ii}}} {h^{jj}}{h^{kk}}{g_{ll}}\left\{ {{\partial _k}\left( {{\partial _j}{u^q}{\partial _i}{u^m}\overline \Gamma  _{mq}^l} \right)} \right. + \Gamma _{kq}^l\left( {\frac{{{\partial ^2}{u^q}}}{{\partial {x_j}\partial {x_i}}} + {\partial _j}{u^k}{\partial _i}{u^m}\overline \Gamma  _{mk}^q - \Gamma _{ij}^m(u){\partial _m}{u^q}} \right) \nonumber\\
&- \Gamma _{ki}^p{\partial _j}{u^q}{\partial _p}{u^m}\overline \Gamma  _{mq}^l{\left. { - \Gamma _{k,j}^p{\partial _i}{u^q}{\partial _p}{u^m}\overline \Gamma  _{mq}^l} \right\}^2}{\rm{dvol_h}} \label{zxcv}
\end{align}
by $\|u\|_{{\mathfrak{H}}^3}$.
Meanwhile (\ref{1.1}), (\ref{1.2}), (\ref{1.3}) imply
\begin{align*}
 {h^{ii}}{h^{jj}}{h^{kk}}{g_{ll}}\left( {\frac{{{\partial ^2}{u^m}}}{{\partial {x_k}\partial {x_i}}}{\partial _j}{u^q}\overline \Gamma  _{mq}^l} \right)^2 &\lesssim {\left| {du} \right|^4} + {\left| {\nabla du} \right|^4}, \\
 {h^{ii}}{h^{jj}}{h^{kk}}{g_{ll}}\left( {\frac{{{\partial ^2}{u^q}}}{{\partial {x_k}\partial {x_j}}}{\partial _i}{u^m}\overline \Gamma  _{mq}^l} \right)^2 &\lesssim {\left| {du} \right|^4} + {\left| {\nabla du} \right|^4}, \\
 {h^{ii}}{h^{jj}}{h^{kk}}{g_{ll}}{\left( {{\partial _j}{u^q}{\partial _i}{u^m}{\partial _k}\overline \Gamma  _{mq}^l} \right)^2} &\lesssim {\left| {du} \right|^6}, \\
 {h^{ii}}{h^{jj}}{h^{kk}}{g_{ll}}{\left( {\Gamma _{k,j}^p{\partial _i}{u^q}{\partial _p}{u^m}\overline \Gamma  _{mq}^l} \right)^2} &\lesssim {\left| {du} \right|^4}.
 \end{align*}
Thus (\ref{jia81}) follows from Sobolev embedding.
The inverse direction of (\ref{jia81}), i.e., (\ref{jia91}) is obtained along the same path. Indeed, Sobolev embedding yields for $l=1,2$
\begin{align}
\|u^l-Q^l\|_{L^{\infty}}\lesssim \sum^2_{k=1}\|u^k-Q^k\|_{H^{2}}.
\end{align}
Thus one has for $l=1,2$
$$\|u^l\|_{L^{\infty}}\lesssim R_0+\|u\|_{\mathcal{H}^2}.
$$
Then the $k=2$ case of (\ref{jia91}) follows by Lemma \ref{hbvcm7} with $R=R_0+C\|u\|_{\mathcal{H}^{2}}$. Further calculations give the $k=3$ case of (\ref{jia91}).
\end{proof}

\begin{remark}\label{jiujin}
As a byproduct of the proof in Lemma \ref{new}, we have for some continuous function $\mathcal{P}:\Bbb R\to \Bbb R$ depending only on $Q$ such that for any $u\in \mathcal{H}^2_{Q}$
$$
\|(u^1,u^2)\|_{L^{\infty}}\lesssim \mathcal{P}(\|u\|_{\mathfrak{H}^2}).
$$
\end{remark}

\subsection{Smoothing effects for heat equations on hyperbolic spaces}

We also need a smoothing effect lemma for the semigroup $e^{z\Delta_{\Bbb H^2}}$ in Section 5.
\begin{lemma}\label{anxin}
Let $z\in \Bbb C$ with $\Re z=\alpha>0$. If $f\in L^{\infty}([0,T];H^2)$ with initial data $f_0$ is a solution to the linear inhomogeneous equation
\begin{align}\label{ch78}
{\partial _t}f - z{\Delta _{{\Bbb H^2}}}f = g,
\end{align}
then we have
$$\|f\|_{L^2_x}\lesssim \|f_0\|_{L^2_x}+ \|g\|_{L^2_t{\dot{H}}^{-1}_x}.$$
\end{lemma}
\begin{proof}
Taking inner product with $f$ on both sides of (\ref{ch78}), the real part yields
\begin{align}
\frac{d}{{dt}}\left\| f \right\|_{{L_x}}^2 - \alpha \left\| {\nabla f} \right\|_{{L_x}}^2 =\Re \left\langle {g,f} \right\rangle.
\end{align}
Then Lemma \ref{anxin} follows by (\ref{uv5}) and
\begin{align*}
\left| {\Re\left\langle {g,f} \right\rangle } \right| \le {\left\| {{{\left( { - \Delta } \right)}^{\frac{1}{2}}}f} \right\|_{L_x^2}}{\left\| {{{\left( { - \Delta } \right)}^{ - \frac{1}{2}}}g} \right\|_{L_x^2}}.
\end{align*}
\end{proof}

The pointwise estimate of the heat kernel in $\mathbb{H}^2$ is obtained by Davies and Mandouvalos \cite{DM}.
\begin{lemma}\label{8.5}
The heat kernel on $\Bbb H^2$ denoted by $K_2(t,\rho)$ satisfies the pointwise estimate
$$
K_2(t,\rho)\sim t^{-1}e^{-\frac{1}{4}t}e^{-\frac{\rho^2}{4t}}e^{-\frac{1}{2}\rho}(1+\rho+t)^{-\frac{1}{2}}(1+\rho).
$$
Particularly, we have the decay estimate
\begin{align}
\|e^{s\Delta_{\Bbb H^2}}f\|_{L^{\infty}_x}&\lesssim e^{-\frac{s}{4}}s^{-1}\|f\|_{L^{1}_x}\label{huhu899}\\
\|e^{s\Delta_{\Bbb H^2}}f\|_{L^p_x}&\lesssim s^{\frac{1}{p}-\frac{1}{r}}\|f\|_{L^{r}_x},\label{huhu89}
\end{align}
where $1\le r\le p\le\infty$.
\end{lemma}
\begin{proof}
(\ref{huhu89}) can be found in \cite{Coding}. (\ref{huhu899}) follows directly by the upper bound for the heat kernel given above.
\end{proof}

The following lemma for the heat semigroup in $\Bbb R^2$ was obtained in Lemma 2.5 of Tao \cite{Tao4}. We remark that the same arguments work in the $\Bbb H^2$ case, because the proof in \cite{Tao4} only uses the decay estimate (\ref{huhu89}) and the self-ajointness of $e^{t\Delta}$, which are also satisfied by $e^{t\Delta_{\Bbb H^2}}$.
\begin{lemma}\label{appendix1}
For $f\in L^2_x$ defined on $\Bbb H^2$, we have
$$\int^{\infty}_0\|e^{t\Delta_{\Bbb H^2}}f\|^2_{L^{\infty}_x}dt\lesssim \|f\|^2_{L^2_x}.
$$
\end{lemma}

\section{The caloric gauge}
In order to study the asymptotic behaviors, we need to rewrite (\ref{1}) under a gauge. Let $\{e_1(t,x),Je_1(t,x)\}$ be an orthonormal frame for $u^*(T\mathbb{H}^2)$, i.e., $\{e_1(t,x),J(u(t,x))e_1(t,x)\}$ spans $T_{u(t,x)}\mathbb{H}^2$ for each $(t,x)\in [0,T]\times \mathbb{H}^2$ and $|e_1(t,x)|=1.$ Let $\psi_i=(\psi^1_i,\psi^2_i)$ for $i=0,1,2$ be the components of $\partial_{t,x}u$ in the frame $\{e_1(t,x),Je_1(t,x)\}$
$$
\psi _i^1 = \left\langle {{\partial _i}u,{e_1}} \right\rangle ,\psi _i^2 = \left\langle {{\partial _i}u,J{e_1}} \right\rangle.
$$
The isomorphism of $\Bbb R^2$ to $\Bbb C$ induces a complex valued function defined by $\phi_i=\psi_i^1+\sqrt{-1}\psi_i^2$. For any function $\varphi:[0,T]\times\Bbb H^2\to \Bbb C$, we associate it with a tangent vector field on $u^*(T\mathbb{H}^2)$ defined by
$$\varphi\longleftrightarrow\varphi e\triangleq \varphi^1e_1+\varphi^2Je_1.
$$ Then $u$ induces a covariant derivative on the trivial complex vector bundle $[0,T]\times\Bbb H^2$ defined by
$$D_i\varphi=\partial_i \varphi+\sqrt{-1}[A_i]^2_1\varphi,
$$
where the induced connection coefficients are defined by $[{A_i}]_1^2 = \left\langle {{\nabla _i}{e_1},J{e_1}} \right\rangle$. For simplicity, we denote $[A_i]^2_1$ by $A_i$ in the following.
It is easy to check the torsion free identity
\begin{align}\label{pknb}
D_i\phi_j=D_j\phi_i.
\end{align}
The commutator identity is given by
\begin{align}\label{commut}
[D_i,D_j]\varphi=\sqrt{-1}(\partial_i A_j-\partial_j A_i)\varphi\longleftrightarrow\mathbf{R}(\partial_iu, \partial_j u)(\varphi e).
\end{align}
We have a gauge freedom to choose the frame  $\{e_1,Je_1\}$. In fact, given any real valued function $\chi:[0,T]\times\Bbb H^2\to\Bbb R$, under the transform $U$ defined by
$$
U\left( {{k_1}{e_1} + {k_2}J{e_1}} \right) = ({k_1}\cos \chi  - {k_2}\sin \chi ){e_1} + \left( {{k_1}\sin \chi  + {k_2}\cos \chi } \right)J{e_1},
$$
we have
\begin{align}
&(e_1,J{e_1})\to (U{e_1},JU{e_1})\nonumber\\
&{A_i} \to \widehat{A}_i\triangleq\left\langle {{\nabla _i}U{e_1},JU{e_1}} \right\rangle  = {\partial _i}\chi  + {A_i}.\label{gauge}
\end{align}

\begin{lemma}
With the notions and notations given above, (\ref{1}) can be written as
\begin{align}\label{jnk}
\phi_t=zh^{ij}D_i\phi_j-zh^{ij}\Gamma^k_{ij}\phi_k,
\end{align}
where $z=\alpha-\sqrt{-1}\beta$.
\end{lemma}
\begin{proof}
We first rewrite the tension field $\tau(u)$ under the gauge. Recall that
\begin{align*}
\tau(u)&=h^{ij}\nabla_i\partial_j u-h^{ij}u_*(\nabla_{\frac{\partial}{\partial x_i}}{\frac{\partial}{\partial x_j}})\\
&= {h^{ij}}{\nabla _i}\left( {\left\langle {{\partial _j}u,{e_1}} \right\rangle {e_1} + \left\langle {{\partial _j}u,J{e_1}} \right\rangle J{e_1}} \right) - \Gamma _{i,j}^k{h^{ij}}{\partial _k}u \\
&= {h^{ij}}\left( {{\partial _i}\psi _j^1{e_1} + \psi _j^1{A_i}J{e_1} + {\partial _i}\psi _j^2J{e_1} - \psi _j^2{A_i}{e_1}} \right) - \Gamma _{i,j}^k{h^{ij}}\psi _k^1{e_1} - \Gamma _{i,j}^k{h^{ij}}\psi _k^2J{e_1} \\
&= {h^{ij}}\left( {{\partial _i}\psi _j^1 - \psi _j^2{A_i}} \right){e_1} + {h^{ij}}\left( {\psi _j^1{A_i} + {\partial _i}\psi _j^2} \right)J{e_1} - \Gamma _{i,j}^k{h^{ij}}\psi _k^1{e_1} - \Gamma _{i,j}^k{h^{ij}}\psi _k^2J{e_1}.
\end{align*}
By the definition, we have
\begin{align*}
 \psi _t^1 &= \left\langle {{\partial _t}u,{e_1}} \right\rangle  = \alpha \left\langle {\tau (u),{e_1}} \right\rangle  -\beta \left\langle {J\tau (u),{e_1}} \right\rangle  \\
 &= \alpha {h^{ij}}\left( {{\partial _i}\psi _j^1 - \psi _j^2{A_i}} \right) - \alpha \Gamma _{i,j}^k{h^{ij}}\psi _k^1 +\beta {h^{ij}}\left( {\psi _j^1{A_i} + {\partial _i}\psi _j^2} \right) - \beta \Gamma _{i,j}^k{h^{ij}}\psi _k^2 \\
 \psi _t^2 &= \left\langle {{\partial _t}u,J{e_1}} \right\rangle  = \alpha \left\langle {\tau (u),J{e_1}} \right\rangle  -\beta \left\langle {J\tau (u),J{e_1}} \right\rangle  \\
 &= \alpha {h^{ij}}\left( {\psi _j^1{A_i} + {\partial _i}\psi _j^2} \right) - \alpha \Gamma _{i,j}^k{h^{ij}}\psi _k^2 - \beta {h^{ij}}\left( {{\partial _i}\psi _j^1 - \psi _j^2{A_i}} \right) + \beta \Gamma _{i,j}^k(u){h^{ij}}\psi _k^1.
\end{align*}
Then the complex valued function $\phi_t$ satisfies
$${\phi _t} = \alpha {h^{ij}}{D_i}{\phi _j} - \alpha {h^{ij}}\Gamma _{i,j}^k(u){\phi _k} - \sqrt { - 1} \beta {h^{ij}}{D_i}{\phi _j} +\sqrt { - 1} \beta {h^{ij}}\Gamma _{i,j}^k(u){\phi _k}.
$$
\end{proof}

The caloric gauge was first introduced by Tao \cite{Tao3} for the wave maps from $\Bbb R^{2+1}$ to $\mathbb{H}^n$. We give the definition of the caloric gauge in our setting.
\begin{definition}\label{pp}
Let  $u(t,x):[0,T]\times \mathbb{H}^2\to \mathbb{H}^2$ be a solution of (\ref{1}) in $C([0,T];\mathcal{H}^3)$. Suppose that the heat flow with $u_0$ as the initial data converges to a harmonic map $Q_{\infty}$ from $\mathbb{H}^2$ to $\mathbb{H}^2$. For a given orthonormal frame $\Xi(x)\triangleq\{\Xi_1(Q(x)),J(Q(x))\Xi_1(Q(x))\}$ which spans the tangent space $T_{Q(x)}\mathbb{H}^2$ for any $x\in \mathbb{H}^2$, a caloric gauge is a tuple consisting of a map  $\widetilde{u}:\Bbb R^+\times [0,T]\times\mathbb{H}^2\to\Bbb H^2$ and an orthonormal frame $\Omega\triangleq\{\Omega_1(\widetilde{u}(s,t,x)),J(\widetilde{u}(s,t,x))\Omega(\widetilde{u}(s,t,x))\}$ such that
\begin{align}\label{muqi}
\left\{ \begin{array}{l}
{\partial _s}\widetilde{u}= \tau (\widetilde{u}) \\
{\nabla _s}{\Omega _1} = 0 \\
\mathop {\lim }\limits_{s \to \infty } {\Omega _1} = {\Xi _1} \\
\end{array} \right.
\end{align}
where the convergence of frames is defined by
\begin{align}\label{convergence}
\left\{ \begin{array}{l}
 \mathop {\lim }\limits_{s \to \infty } \widetilde{u}(s,t,x) = Q(x) \\
 \mathop {\lim }\limits_{s \to \infty } \left\langle {{\Omega _1}(s,t,x),{\Theta _1}(\widetilde{u}(s,t,x))} \right\rangle  = \left\langle {{\Xi _1}(Q(x)),{\Theta _1}(Q(x))} \right\rangle  \\
 \mathop {\lim }\limits_{s \to \infty } \left\langle {{\Omega _1}(s,t,x),{\Theta _2}(\widetilde{u}(s,t,x))} \right\rangle  = \left\langle {{\Xi _1}(Q(x)),{\Theta _2}(Q(x))} \right\rangle  \\
 \end{array} \right.
\end{align}
\end{definition}

The remaining part of this section is devoted to the existence of the caloric gauge. The equation of the heat flow is given by
\begin{align}\label{8.29.1}
\left\{ \begin{array}{l}
 {\partial _s}u = \tau (u) \\
 u(s,x)\upharpoonright_{s=0} = {u_1}(x) \\
 \end{array} \right.
\end{align}
We recall the definition of the energy density $e$,
$$
e(u)=\frac{1}{2}|du|^2.
$$

The following two lemmas are essentially due to Eells, Sampson \cite{ES} and Li, Tam \cite{LT}.
\begin{lemma}\label{8.44}
Given initial data $u_1:\Bbb H^2\to\Bbb H^2$ with bounded energy density, suppose that $\tau(u_1)\in L^p_x$ for some $p>2$ and the image of $\Bbb H^2$ under the map $u_1$ is contained in a compact subset of $\Bbb H^2$. Then the heat flow equation (\ref{8.29.1}) has a global solution $u$, moreover the energy density $e(u)$ satisfies
\begin{align}
&(\partial_s-\Delta_{\mathbb{H}^2})|du|^2+2|\nabla d(u)|^2\le Ke(u)\label{8.4}\\
&(\partial_s-\Delta_{\mathbb{H}^2})|\partial_s u|^2+2|\nabla \partial_su|^2\le 0.\label{8.3}
\end{align}
\end{lemma}

\begin{corollary}
If $u$ is the solution in Lemma \ref{8.44}, then for some $C>0$
\begin{align}
(\partial_s-\Delta_{\mathbb{H}^2})|\partial_s u|&\le 0\label{VI4}\\
(\partial_s-\Delta_{\mathbb{H}^2})(|du|e^{-Cs})&\le 0.\label{uu}
\end{align}
\end{corollary}
\begin{proof}
Direct calculations give
\begin{align}
{\partial _s}\left| {{\partial _s}u} \right| = \frac{1}{2}\frac{1}{{\left| {{\partial _s}u} \right|}}{\partial _s}{\left| {{\partial _s}u} \right|^2},\mbox{  }\mbox{  }{\Delta _{{\Bbb H^2}}}\left| {{\partial _s}u} \right| = \frac{1}{2}\left[ {{\Delta _{{\Bbb H^2}}}{{\left| {{\partial _s}u} \right|}^2} -2 {{\left| {\nabla \left| {{\partial _s}u} \right|} \right|}^2}} \right]\frac{1}{{\left| {{\partial _s}u} \right|}}.
\end{align}
Then the diamagnetic inequality
$|\nabla |\partial_su||\le |\nabla \partial_su|$ and (\ref{8.3}) yield
\begin{align*}
 {\partial _s}\left| {{\partial _s}u} \right| - {\Delta _{{\Bbb H^2}}}\left| {{\partial _s}u} \right| &= \frac{1}{{2\left| {{\partial _s}u} \right|}}\left( {{\partial _s}{{\left| {{\partial _s}u} \right|}^2} - {\Delta _{{\Bbb H^2}}}{{\left| {{\partial _s}u} \right|}^2} +2 {{\left| {\nabla \left| {{\partial _s}u} \right|} \right|}^2}}\right) \\
&\le {\left| {{\partial _s}u} \right|}\left( { - {{\left| {\nabla {\partial _s}u} \right|}^2} + {{\left| {\nabla \left| {{\partial _s}u} \right|} \right|}^2}} \right)
\end{align*}
thus proving (\ref{VI4}).
To prove (\ref{uu}), similar calculations, the diamagnetic inequality
$|\nabla |du||\le |\nabla du|$ and (\ref{8.4}) imply
\begin{align*}
 {\partial _s}|du| - {\Delta _{{\Bbb H^2}}}|du|  &= \frac{1}{{2|du|}}\left( {{\partial _s}|du|^2 - {\Delta _{{\Bbb H^2}}}|du|^2 +2 {{\left| {\nabla |du|} \right|}^2}} \right) \\
 &\le \frac{1}{{2|du|}}\left( { - 2{{\left| {\nabla du} \right|}^2} + K|du|^2 + 2{{\left| {\nabla \left| {du} \right|} \right|}^2}} \right) \\
 &\le \frac{1}{{2|du|}}K|du|^2.
\end{align*}
Let $C=\frac{K}{2}$, we obtain (\ref{uu}).
\end{proof}

It is known in the heat flow literature that Harnack inequality for the linear heat equation is useful to obtain $L^{\infty}$ bounds. The Harnack type inequality for complete manifolds was initially proved by Li and Yau \cite{LY}. The following form of Harnack inequality which is convenient in our case was obtained by Li and Xu \cite{LX}.
\begin{lemma}\label{kk}
If $M$ is a n-dimensional complete Riemannian manifold with $Ricci_M\ge -k$, and if $f(x,t):M\times(0,\infty)\to \Bbb R^+$ is a positive solution of the linear heat equation on $M$, then for $\forall x_1,x_2\in M$, $0<t_1<t_2<\infty$, the following inequality holds:
 $$
 u(x_1,t_1)\le u(x_2,t_2)B_1(t_1.t_2)exp\big[\frac{dist^2(x_2,x_1)}{4(t_2-t_1)}(1+B_2(t_1,t_2))\big],
 $$
where $dist(x_1,x_2)$ is the distance between $x_1$ and $x_2$, $B_1(t_1,t_2)=\big(\frac{e^{2kt_2}-2kt_2-1}{e^{2kt_1}-2kt_1-1}\big)^{\frac{n}{4}}$, and $B_2(t_1,t_2)=\frac{t_2coth(kt_2)-t_1coth(kt_1)}{t_2-t_1}.$
\end{lemma}

Now we consider the heat flow from $\mathbb{H}^2$ to $\mathbb{H}^2$ with a parameter
\begin{align}\label{8.29.2}
\left\{ \begin{array}{l}
 {\partial _s}\widetilde{u} = \tau (\widetilde{u}) \\
 \widetilde{u}(s,t,x) \upharpoonright_{s=0}= u(t,x) \\
 \end{array} \right.
\end{align}
\begin{lemma}\label{11.1}
If $u(t,x)$ is a solution to (\ref{1}) in $C([0,T];\mathcal{H}^3)$, then there exists a harmonic map $Q_{\infty}$ such that as $s\to\infty$,
$$
\mathop {\lim }\limits_{s \to \infty } \mathop {\sup }\limits_{(x,t) \in {\mathbb{H}^2} \times [0,T]}
dist_{\mathbb{H}^2}(\widetilde{u}(s,x,t),Q_{\infty}(x))=0.
$$
\end{lemma}
\begin{proof}
The global existence of $\widetilde{u}$ is due to Lemma \ref{8.44}, Remark \ref{jiujin}, the embedding $\mathcal{H}^1\hookrightarrow L^{p}$ for $p\in[2,\infty)$ and diamagnetic inequality.
By (\ref{8.3}) and the maximum principle, we have the bound
\begin{align}\label{9.1}
|\partial_s \widetilde{u}|^2\le e^{s\Delta_{\mathbb{H}^2}}|\partial_s\widetilde{u}(0,t,x)|^2.
\end{align}
Thus (\ref{huhu899}) shows
\begin{align}\label{10.112}
\mathop {\sup }\limits_{x \in {\mathbb{H}^2}} \left| {{\partial _s}\widetilde{u}(s,t,x)} \right|^2 \le {s^{ - 1}}{e^{ - \frac{1}{4}s}}\int_{{\mathbb{H}^2}} {{{\left| {{\partial _s}\widetilde{u}(0,t,x)} \right|}^2}{\rm{dvol_h}}}.
\end{align}
Then (\ref{1}) and (\ref{8.29.2}) yield
$$
\mathop {\sup }\limits_{(x,t) \in {\mathbb{H}^2}\times[0,T]} \left| {{\partial _s}\widetilde{u}(s,t,x)} \right| \le {s^{ -\frac{1}{2}}}{e^{ - \frac{1}{8}s}}\int_{{\mathbb{H}^2}} {{{\left| {{\partial _t}u(t,x)} \right|}^2}{\rm{dvol_h}}}\le C(T){s^{ - 1}}{e^{ - \frac{1}{8}s}}.
$$
Therefore for any $1<s_1<s_2<\infty$  we have
$${d_{{\mathbb{H}^2}}}(\widetilde{u}({s_1},t,x),\widetilde{u}({s_2},t,x)) \le C\int_{{s_1}}^{{s_2}} {{e^{ - \frac{1}{8}s}}ds},$$
which implies $\widetilde{u}(s,t,x)$ converges uniformly on $(t,x)\in[0,T]\times \mathbb{H}^2$ to some map $Q_{\infty}(t,x)$. By Theorem 5.2 in \cite{LT}, for any fixed $t$, $Q_{\infty}(t,x)$ is a harmonic map form $\Bbb H^2\to\Bbb H^2$ with respect to $x$. It remains to prove $Q_{\infty}(t,x)$ is indeed independent of $t$. We consider the evolution of $|\partial_t u|^2$ with respect to $s$. In fact, $|\partial_t \widetilde{u}|^2$ satisfies
\begin{align}\label{10.1}
(\partial_s-\Delta_{\mathbb{H}^2})|\partial_t \widetilde{u}|^2\le -|\nabla \partial_t \widetilde{u}|^2-\mathbf{R}(\widetilde{u})(\nabla \widetilde{u},\partial_t \widetilde{u},\nabla \widetilde{u}, \partial_t \widetilde{u})\le 0.
\end{align}
Hence the maximum principle and (\ref{huhu899}) imply
\begin{align}\label{sdf}
\mathop {\sup }\limits_{x \in {\mathbb{H}^2}} {\left| {{\partial _t}\widetilde{u}(s,t,x)} \right|^2} \le {s^{ - 1}}{e^{ - \frac{1}{4}s}}\int_{{\Bbb H^2}} {{{\left| {{\partial _t}\widetilde{u}(0,t,x)} \right|}^2}{\rm{dvol_h}}}.
\end{align}
Consequently, we obtain for $0\le t_1<t_2\le T$,
$${d_{{\mathbb{H}^2}}}(\widetilde{u}(s,{t_1},x),\widetilde{u}(s,{t_2},x)) \le \int_{{t_1}}^{{t_2}} {\left| {{\partial _t}\widetilde{u}(s,t,x)} \right|} dt \le C{s^{ - \frac{1}{2}}}{e^{ - \frac{s}{8}}}({t_2} - {t_1}).
$$
Let $s\to\infty$, we conclude ${d_{{\mathbb{H}^2}}}(Q_{\infty}({t_1},x),Q_{\infty}({t_2},x)) = 0$, thus finishing the proof.
\end{proof}

In order to obtain a global bound independent of $s$ for the quantities related to the heat flow (\ref{8.29.2}), we need to get estimates which only depend on the energy of $u(t,x)$.
\begin{lemma}
Let $u$ solve $(\ref{1})$ in $C([0,T];\mathcal{H}^3)$. If $\widetilde{u}$ is the solution to $(\ref{8.29.2})$ with initial data $u(t,x)$, then we have
\begin{align}
\int_0^\infty  {\left\| { {{\partial _s}\widetilde{u}(s,t,x)} } \right\|} _{L_x^\infty }^2ds &\lesssim \left\| {{{\tau}(u(t,x))}} \right\|_{L_x^2}^2 \label{9.2}\\
\int_0^1 {\left\| {\widetilde{e}(s,t,x)} \right\|} _{L_x^\infty }ds &\lesssim \left\| {{e_0}} \right\|_{L_x^1} \label{9.3}\\
\mathop {\sup }\limits_{(s,t) \in [\lambda,\infty ) \times [0,T]} {\left\| {\widetilde{e}(s,t,x)} \right\|_{L_x^\infty }} &\lesssim c(\lambda)\left\| {{e_0}} \right\|_{L_x^1},\label{3.14a}
\end{align}
where $\widetilde{e}$ and $e_0$ are the energy density of $\widetilde{u}$ and $u_0$ respectively.
\end{lemma}
\begin{proof}
We first prove (\ref{9.3}). (\ref{uu}) with the maximum principle gives
$$
\sqrt{\widetilde{e}(s,t,x)}\le e^{Cs}e^{s\Delta_{\mathbb{H}^2}}\sqrt{\widetilde{e}(0,t,x)}.
$$
Then Lemma \ref{appendix1} yields
$$
\int^1_0\|\widetilde{e}(s,t,x)\|_{L^{\infty}_x}ds\lesssim e^{C}\int^1_0\|e^{s\Delta_{\mathbb{H}^2}}\sqrt{\widetilde{e}(0,t,x)}\|^2_{L^{\infty}_x}ds\lesssim \|\sqrt{\widetilde{e}(0,t,x)}\|^2_{L^2_x}.
$$
Thus (\ref{9.3}) is obtained.
(\ref{9.2}) is a direct consequence of (\ref{VI4}) and Lemma \ref{appendix1}. The remained (\ref{3.14a}) needs the Harnack inequality stated in Lemma \ref{kk}. Fix an arbitrary $s_1\ge \lambda$. Suppose that $w(s,t,x):\Bbb R^+\times [0,T]\times \mathbb{H}^2\to \Bbb R$ is a solution to the linear heat equation
\begin{align}\label{hh}
\left\{ \begin{array}{l}
 {\partial _s}w - {\Delta _{{\Bbb H^2}}}w = 0, \\
 w(s,t,x)\upharpoonright_{s=s_1-\frac{\lambda}{2}} = {e^{ - {s_1}K+\frac{\lambda}{2}K}}\widetilde{e}({s_1}-\frac{\lambda}{2},t,x).
 \end{array} \right.
\end{align}
Elementary calculus especially the mean-value formula gives for $s\ge\frac{\lambda}{2}$,
$$ B_1(s,s+1)\lesssim c(\lambda),\mbox{  }B_2(s,s+1)\lesssim c(\lambda).
$$
Therefore, we infer from Lemma \ref{kk} that for any $x_1,x_2\in \mathbb{H}^2$ with $dist(x_1,x_2)\le 1$, $s\ge s_1$, we have a uniform bound
\begin{align}\label{a1.1}
w(s,x_1)\lesssim c(\lambda)w(s+1,x_2).
\end{align}
On the other hand, we have
$$(\partial_s-\Delta_{\mathbb{H}^2})(e^{-sK}\widetilde{e}(s)-w(s))\le 0, \mbox{ }w(s_1-\frac{\lambda}{2})=e^{-s_1K+\frac{\lambda}{2}K}\widetilde{e}(s_1-\frac{\lambda}{2}),
$$
thus by maximum principle for any $s\ge s_1-\lambda/2$,
$$\widetilde{e}(s)\le e^{sK}w(s).
$$
Therefore (\ref{a1.1}) yields a bound for $\widetilde{e}$
\begin{align}\label{9,8}
\widetilde {e}(s_1,t,x_1)\le e^{s_1K}w(s_1,t,x_1)\lesssim c(\lambda)e^{s_1K}w(s_1+1,t,x_2).
\end{align}
Integrating (\ref{9,8}) on the geodesic ball centered at $x_1$ of radius 1 with respect to $x_2$, one has
$$
\widetilde {e}(s_1)\lesssim c(\lambda)e^{s_1K}\int_{\mathbb{H}^2}w(s_1+1,t,x){\rm{d{vol}_h}}.
$$
Since $e^{s\Delta_{\mathbb{H}^2}}$ is bounded on $L^2$ then we have from (\ref{hh}) that
$$
\widetilde {e}(s_1,t,x)\lesssim c(\lambda)\int_{\mathbb{H}^2}\widetilde{e}(s_1-\frac{\lambda}{2},t,y){\rm{d{vol}_h}}.
$$
Then (\ref{3.14a}) follows by the non-increasing of the energy.
\end{proof}

\begin{remark}
By (\ref{VI4}) and (\ref{10.1}) and some calculations we have
\begin{align}
|\partial_s\tilde{u}(s,t,x)|&\lesssim e^{s\Delta_{\Bbb H^2}}\big(|\partial_t u(t,x)|\big),\label{VI3}\\
|\partial_t\tilde{u}(s,t,x)|&\le e^{s\Delta_{\Bbb H^2}}\big(|\partial_t u(t,x)|\big).\label{VI9}
\end{align}
\end{remark}

Now we prove the existence of the caloric gauge in Definition \ref{pp}.
\begin{proposition}\label{3.3}
Given any solution $u(t,x)$ of (\ref{1}) in $C([0,T]; \mathcal{H}^3)$, suppose that the limit harmonic map of the heat flow (\ref{8.29.2}) with initial data $u_0$ is $Q_{\infty}(x)$. For any fixed frame $\Xi\triangleq\{\Xi_1(Q_{\infty}(x)),J\Xi_1(Q_{\infty}(x))\}$, there exists a unique corresponding caloric gauge defined in Defintion \ref{pp}.
\end{proposition}
\begin{proof}
We first show the existence part. Choose an arbitrary orthonormal frame $E_0(t,x)\triangleq\{e_1(t,x),Je_1(t,x)\}$ such that $E_0(t,x)$ spans the tangent space $T_{u(t,x)}{\mathbb{H}^2}$ for each $(t,x)\in [0,T]\times \mathbb{H}^2$. This kind of frame does exist, in fact we have a global orthonormal frame for $\mathbb{H}^2$ defined by (\ref{frame}). Then evolving (\ref{8.29.2}) with initial data $u(t,x)$, we have from Lemma \ref{11.1} that there exists a harmonic map $Q_{\infty}:\mathbb{H}^2\to \mathbb{H}^2,$ such that $\widetilde{u}(s,t,x)$ converges to $Q_{\infty}$ uniformly for $(t,x)\in[0,T]\times \mathbb{H}^2$ as $s\to\infty$. Meanwhile, we evolve $E_0$ in $s$ according to
\begin{align}\label{11.2}
\left\{ \begin{array}{l}
{\nabla _s}{\Omega _1}(s,t,x) = 0 \\
{\Omega _1}(s,t,x)\upharpoonright_{s=0} = {e_1}(t,x) \\
\end{array} \right.
\end{align}
Denote the evolved frame as $E_s\triangleq \{\Omega_1(s,t,x),J\Omega_1(s,t,x)\}$. Then $\Omega_2\triangleq J\Omega_1(s,t,x)$ satisfies
${\nabla _s}{\Omega _2}(s,t,x) = 0$ as well. We claim that there exists some orthonormal frame $E_{\infty}\triangleq\{e_1(\infty,t,x),Je_1(\infty,t,x)\}$ which spans $T_{Q_{\infty}(x)}\mathbb{H}^2$ for each $(t,x)\in [0,T]\times\mathbb{H}^2$ such that
\begin{align}\label{pl}
\mathop {\lim }\limits_{s \to \infty }{\Omega_1(s,t,x)}  = e_1(\infty,t,x).
\end{align}
Indeed, by the definition of the convergence of frames given in (\ref{convergence}) and the fact $\widetilde{u}(s,t,x)$ converges to $Q_{\infty}(x)$, it suffices to show for some scalar function $a_i:[0,T]\times \mathbb{H}^2\to \Bbb R$
\begin{align}\label{aw}
\mathop {\lim }\limits_{s \to \infty } \left\langle {{\Omega_1}(s,t,x),{\Theta _i}(\widetilde{u}(s,t,x))} \right\rangle  = a_i(t,x).
\end{align}
Direct calculations give
$$
\left| \nabla _s \Theta_i(\widetilde{u}(s,t,x))\right| \lesssim \left| {{\partial _s}\widetilde{u}} \right|,
$$
which combined with (\ref{10.112}) and $\nabla_s\Omega=0$ implies that for $s>1$
$$\big|{\partial _s}\left\langle {{\Omega_1}(s,t,x),{\Theta _i}(\widetilde{u}(s,t,x))} \right\rangle \big|\lesssim C(T)e^{-\frac{1}{8}s}.
$$
Hence $(\ref{aw})$ holds for some $a_i(t,x)$, thus verifying (\ref{pl}). It remains to adjust the initial frame $E_0$ to make the limit frame $E_{\infty}$ coincide with the given frame $\Xi$. This can be achieved by the gauge transform invariance illustrated in Section 2.1. Indeed, since for any $U:[0,T]\times\mathbb{H}^2\to SO(2)$ and the solution $\widetilde{u}(s,t,x)$ to (\ref{8.29.2}), we have $\nabla_s U(t,x)\Omega(s,t,x)=U(t,x)\nabla_s\Omega(s,t,x)$, then the following gauge symmetry holds
\begin{align*}
 {E_0}\triangleq\left\{ {{e_1}(t,x),J{e_1}(t,x)} \right\} &\mapsto {{E'}_0}\triangleq\left\{ {U(t,x){e_1}(t,x),JU(t,x){e_1}(t,x)} \right\} \\
 {E_s}\triangleq\left\{ {{\Omega_1}(s,t,x),J{\Omega_1}(s,t,x)} \right\} &\mapsto {{E'}_s}\triangleq\left\{ {U(t,x){\Omega_1}(s,t,x),JU(t,x){\Omega_1}(s,t,x)} \right\}.
\end{align*}
Therefore choosing $U(t,x)$ such that $U(t,x)E_{\infty}=\Xi$, where $E_{\infty}$ is the limit frame obtained by (\ref{pl}), suffices for our purpose. The uniqueness of the gauge follows from the identity
$$
\frac{d}{{ds}}\left\langle {{\Psi _1} - {\Psi _2},{\Psi _1} - {\Psi _2}} \right\rangle  = 0,
$$
where $(\Psi _1,J\Psi _1)$ and $(\Psi _2,J\Psi _2)$ are two caloric gauges satisfying (\ref{muqi}).
\end{proof}

For any given complex valued functions $\phi,\varphi,\psi$ defined on $[0,T]\times\Bbb H^2$, define the $\wedge$ operator by
\begin{align*}
\phi\wedge\varphi=\phi^1\varphi^2-\phi^2\varphi^1.
\end{align*}
Then (\ref{commut}) reduces to
\begin{align}\label{commut1}
(\partial_i{A_j}-\partial_jA_i)= \phi_i\wedge\phi_j.
\end{align}

The following lemma gives the bounds for the connection coefficients matrix $A_{x,t}$.
\begin{lemma}
Suppose that $\Omega(s,t,x)$ is the caloric gauge constructed in Proposition \ref{3.3}, then we have for $i=1,2$
\begin{align}
&\mathop {\lim }\limits_{s \to \infty } {A_i}(s,t,x) =\left\langle {{\nabla _i}{\Xi_1 }(x),J{\Xi_1}(x)} \right\rangle\label{kji}\\
&\mathop {\lim }\limits_{s \to \infty } {A_t}(s,t,x) = 0\label{kji22}
\end{align}
Particularly, we have for $i=1,2$, $s>0$,
\begin{align}
&{A_i}(s,t,x)\sqrt{h^{ii}}(x) = \int_s^\infty \sqrt{ h^{ii}(x)} \phi_s\wedge\phi _i ds' + { \sqrt{h^{ii}(x)}}\left\langle {{\nabla _i}{\Xi_1}(x),J{\Xi_1}(x)} \right\rangle.\label{edf}\\
&{A_t}(s,t,x)=\int^{\infty}_s\phi_s\wedge\phi_tds'.\label{edf22}
\end{align}
Moreover, let $\Xi(x)=\Theta(Q_{\infty})$ in Proposition \ref{3.3}, we have the bounds for $A_{x,t}$:
\begin{align}
 {\left\| {\sqrt{h^{ii}}{A_i}(t,s,x)} \right\|_{L_s^\infty L_x^\infty ({\Bbb R^ + } \times {\Bbb H^2})}} &\lesssim \|du\|_{L^2_x} + {\left\| {du} \right\|_{L_x^2}}{\left\| {{\partial _t}u(t,x)} \right\|_{L_x^2}} \label{3.30}\\
 {\left\| {{A_t}(t,s,x)} \right\|_{L_s^\infty L_x^\infty ({\Bbb R^ + } \times {\Bbb H^2})}} &\lesssim \left\| {{\partial _t}u(t,x)} \right\|_{L_x^2}^2.\label{VI2}
\end{align}
\end{lemma}
\begin{proof}
Direct calculations give
\begin{align}\label{zfg54xc}
 \left|{\sqrt{h^{ii}(x)}}\left\langle {{\nabla _i}{\Theta_1}(Q_{\infty}(x)),J{\Xi_1}(x)}\right\rangle\right|&\le |{\nabla}{\Theta_1}(Q_{\infty}(x))|\le|dQ_{\infty}(x)|.
\end{align}
Since $u(s,t,x)$ converges to $Q_{\infty}(x)$ in $C^1_{loc}$ and $\|d\widetilde{u}(s,t,x)\|_{L^{\infty}}\lesssim\|du(t,x)\|_{L^2_x}$ uniformly for $s\ge 1$ shown by (\ref{3.14a}), we obtain
\begin{align}\label{zfg55xc}
\|dQ_{\infty}(x)\|_{L^{\infty}}\lesssim \|du(t,x)\|_{L^2_x}.
\end{align}
Thus (\ref{zfg54xc}), (\ref{zfg55xc}) give
\begin{align}\label{zfg57xc}
 \left|{\sqrt{h^{ii}(x)}}\left\langle {{\nabla _i}{\Theta_1}(Q_{\infty}(x)),J{\Theta_1}(Q_{\infty}(x))}\right\rangle\right|&\le \|du(t,x)\|_{L^2_x}.
\end{align}
First, we prove (\ref{edf}) under the assumption of (\ref{kji}). From the identity (\ref{commut1}), we obtain
\begin{align}\label{zxc}
\partial_s A_i-\partial_i A_s=\phi_s \wedge\phi_i,
\end{align}
then (\ref{edf}) follows immediately from $A_s=0$, (\ref{kji}) and the fundamental theorem of calculus.
Second, we prove (\ref{kji}).
Multiplying (\ref{zxc}) with $\sqrt{h^{ii}(x)}$, since $\sqrt{h^{ii}(x)}|\partial_i\widetilde{u}|\le|d\widetilde{u}|$, then by (\ref{3.14a}), we infer from $A_s=0$, (\ref{10.112}) that have for $s\ge \lambda$
\begin{align}\label{cvb2}
|\partial_s(\sqrt{h^{ii}}(x) A_i)|\lesssim c(\lambda)\|\partial_tu(t,x)\|_{L^2_x}e^{-\frac{s}{8}}.
\end{align}
Therefore, we infer from (\ref{cvb}) and (\ref{cvb2}) that for some $a_1(t,x)$ and $a_2(t,x)$
\begin{align}\label{84fr}
\left| {\sqrt{{h^{ii}}}(x){A_i}(s,t,x) - {a_i}(t,x)} \right| \lesssim {e^{ - \frac{s}{8}}}.
\end{align}
Recall the facts that $\widetilde{u}(s,t,x)$ converges to $Q_{\infty}(x)$ in the $C^1_{loc}$ norm as $s\to\infty$  (Theorem 4.3 of \cite{LT}), and the frame $\Omega$ converges to $\Xi$ in the sense of (\ref{convergence}). Then by the definition of $A_i$, in order to prove (\ref{kji}) it suffices to verify
\begin{align}\label{goupigu}
\left\langle {{\nabla _i}{\Omega _j},{\Theta _k}(\widetilde{u})} \right\rangle  \to \left\langle {{\nabla _i}{\Xi _j},{\Theta _k}(Q_{\infty})} \right\rangle.
\end{align}
By the identity
$$\left\langle {{\nabla _i}{\Omega _j},{\Theta _k}(\widetilde{u})} \right\rangle  = {\partial _i}\left\langle {{\Omega _j},{\Theta _k}(\widetilde{u})} \right\rangle  - \left\langle {{\Omega _j},{\nabla _i}{\Theta _k}(\widetilde{u})} \right\rangle,
$$
for any $\varphi\in C^{\infty}_c(\Bbb H^2;\Bbb R)$ we have
\begin{align}\label{cvb}
\int_{{\Bbb H^2}}\sqrt{{h^{ii}}}(x) \left\langle {{\nabla _i}{\Omega _j},{\Theta _k}(\widetilde{u})} \right\rangle \varphi (x){\rm{dvol_h}} &= \int_{{\Bbb H^2}}\sqrt{{h^{ii}}}(x) {\varphi (x)} {\partial _i}\left\langle {{\Omega _j},{\Theta _k}(\widetilde{u})} \right\rangle {\rm{dvol_h}}\nonumber\\
&- \int_{{\Bbb H^2}} \sqrt{{h^{ii}}}(x){\varphi (x)} \left\langle {{\Omega _j},{\nabla _i}{\Theta _k}(\widetilde{u})} \right\rangle {\rm{dvol_h}}.
\end{align}
By integration by parts , $\widetilde{u}(s,t,x)$ converges to $Q_{\infty}(x)$ in the $C^0$ norm as $s\to\infty$ and (\ref{convergence}), the first term in the right hand side of (\ref{cvb}) converges to
$$\int_{{\Bbb H^2}} \sqrt{{h^{ii}}}(x){\varphi (x)} {\partial _i}\left\langle {{\Xi _j},{\Theta _k}(Q_{\infty})} \right\rangle {\rm{dvol_h}}.
$$
By the explicit expression of $\Theta$ and the fact that $\widetilde{u}(s,t,x)$ converges to $Q_{\infty}(x)$ in the $C^1_{loc}$ norm, the second term in the right hand side of (\ref{cvb}) converges to
$$\int_{{\Bbb H^2}} \sqrt{h^{ii}}{\varphi (x)\left\langle {{\Xi _j},{\nabla _i}{\Theta _k}(Q_{\infty})} \right\rangle {\rm{dvol_h}}}.
$$
Then (\ref{goupigu}) holds in the distribution sense. Thus (\ref{84fr}) implies (\ref{kji}).\\
Third, we turn to (\ref{kji22}). We claim it suffices to prove
\begin{align}
\mathop {\lim }\limits_{s \to \infty } {\partial _t}\left\langle {\Omega_1,{\Theta _1}} \right\rangle  = 0,\label{1188}\\
\mathop {\lim }\limits_{s \to \infty } {\partial _t}\left\langle {\Omega_1,J{\Theta _1}} \right\rangle  = 0.\label{11881}
\end{align}
Indeed, (\ref{kji22}) is equivalent to
$$\mathop {\lim }\limits_{s \to \infty } \left\langle {{\nabla _t}\Omega_1 ,{\Theta _1}} \right\rangle  = 0,\mbox{  }\mathop {\lim }\limits_{s \to \infty } \left\langle {{\nabla _t}\Omega_1 ,J{\Theta _1}} \right\rangle  = 0.$$
Since (\ref{sdf}) shows
$$
\|\partial_t\widetilde{u}\|_{L^{\infty}}\lesssim e^{-\frac{s}{8}}\|\partial_tu\|_{L^2_x},
$$
and direct calculation gives
\begin{align}
&{\partial _t}\left\langle {{\Omega _1},{\Theta _i}} \right\rangle  = \left\langle {{\nabla _t}{\Omega _1},{\Theta _i}} \right\rangle  + \left\langle {{\Omega _1},{\nabla _t}{\Theta _i}} \right\rangle  \label{3aw}\\
&{\nabla _t}{\Theta _1} = {\nabla _t}\left( {\frac{\partial }{{\partial {y_1}}}{e^{{\widetilde{u}^2}}}} \right) = {\partial _t}{\widetilde{u}^2}{e^{{\widetilde{u}^2}}}\frac{\partial }{{\partial {y_1}}} + {e^{{\widetilde{u}^2}}}{\partial _t}{\widetilde{u}^k}\overline{\Gamma }_{k1}^l\frac{\partial }{{\partial {y_l}}}\nonumber\\
&{\nabla _t}{\Theta _2} = {\nabla _t}\left( {\frac{\partial }{{\partial {y_2}}}} \right) = {\partial _t}{\widetilde{u}^k}\overline{\Gamma }_{k2}^l\frac{\partial }{{\partial {y_l}}},\nonumber
\end{align}
we conclude that
$$\left| {{\partial _t}\left\langle {{\Omega _1},{\Theta _i}} \right\rangle  - \left\langle {{\nabla _t}{\Omega _1},{\Theta _i}} \right\rangle } \right| \le \left| {{\nabla _t}{\Theta _i}} \right| \le \left| {{\partial _t}\widetilde{u}} \right|,
$$
and further
$$\mathop {\lim }\limits_{s \to \infty } \left| {{\partial _t}\left\langle {{\Omega _1},{\Theta _i}} \right\rangle  - \left\langle {{\nabla _t}{\Omega _1},{\Theta _i}} \right\rangle } \right| = 0.
$$
Therefore it suffices to prove (\ref{1188}) and (\ref{11881}). Since it has been verified in Proposition \ref{3.3} that
$$\mathop {\lim }\limits_{s \to \infty } \left\langle {{\Omega _1},{\Theta _i}} \right\rangle  = \mathop {\lim }\limits_{s \to \infty } \left\langle {{\Xi _1}(Q(x)),{\Theta _i}(Q(x))} \right\rangle,
$$
in order to prove (\ref{1188}), (\ref{11881}), it suffices to show ${\partial _t}\left\langle {{\Omega _1},{\Theta _i}} \right\rangle$ converges uniformly on any given compact subset of $[0,T]\times\Bbb H^2$ as $s\to\infty$. By the definition of $A_t$, $A_s=0$, $\nabla_s\Omega=0$, one has
\begin{align}\label{kku90}
{\partial _s}\left\langle {{\nabla _t}{\Omega _1},{\Theta _i}} \right\rangle  = {\partial _s}\left\langle {{A_t}{\Omega _2},{\Theta _i}} \right\rangle  = {\partial _s}{A_t}\left\langle {{\Omega _2},{\Theta _i}} \right\rangle  + {A_t}\left\langle {{\Omega _2},{\nabla _s}{\Theta _i}} \right\rangle.
\end{align}
Meanwhile, (\ref{commut1}) shows
$${\partial _s}{A_t}=\phi_s\wedge\phi_t.$$
Then we infer from (\ref{10.112}), (\ref{sdf}) that
$$|{\partial _s}{A_t}|\lesssim e^{-\frac{s}{8}}\|\partial_tu\|^2_{L^2_x}.$$
Hence $A_t$ converges uniformly on $[0,T]\times\Bbb H^2$, and thus uniformly bounded w.r.t., $(t,s,x)\in[0,T]\times[1,\infty)\times\Bbb H^2$.    Then by (\ref{kku90}) and $|{{\nabla _s}{\Theta _i}}|\le|\partial_s\widetilde{u}|$, we obtain
$$\left| {{\partial _s}\left\langle {{\nabla _t}{\Omega _1},{\Theta _i}} \right\rangle } \right| \le \left| {{\partial _s}{A_t}} \right| + \left| {{A_t}} \right|\left| {{\partial _s}\widetilde{u}} \right| \le C(T){e^{ - \frac{s}{8}}}.
$$
Therefore $\left\langle {{\nabla _t}{\Omega _1},{\Theta _i}} \right\rangle$ converges as $s\to\infty$. Hence the convergence of ${\partial _t}\left\langle {{\Omega _1},{\Theta _i}} \right\rangle$
as $s\to\infty$ follows from (\ref{3aw}), $\left| {{\nabla _t}{\Theta _i}} \right|\le |\partial_t\widetilde{u}|\lesssim e^{-\frac{s}{8}}$. Thus (\ref{1188}) is obtained, and similar arguments yield (\ref{11881}) which ends the proof of (\ref{kji22}).\\
Forth, we verify (\ref{VI2}). Lemma \ref{appendix1}, (\ref{VI3}) and (\ref{VI9}) show
\begin{align*}
\int_0^\infty  {\left\| {{\partial _s}\tilde{u}(s,t,x)} \right\|_{L_x^\infty }^2ds \le } \left\| {{\partial _t}u(t,x)} \right\|_{L_x^2}^2,
\mbox{  }\int_0^\infty  {\left\| {{\partial _t}\tilde{u}(s,t,x)} \right\|_{L_x^\infty }^2ds \le } \left\| {{\partial _t}u(t,x)} \right\|_{L_x^2}^2.
\end{align*}
Thus (\ref{VI2}) follows by the definition of $\phi_s$, $\phi_t$ and H\"older inequality. It remains to prove (\ref{3.30}). By (\ref{sdf}) for $s\ge1$ we have
$${\left\| {{\partial _s}u(s,t,x)} \right\|_{L_x^\infty }} \le {e^{ - \frac{s}{8}}}{\left\| {{\partial _t}u(t,x)} \right\|_{L_x^2}},
$$
which combined with (\ref{3.14a}) yields
\begin{align}\label{VI5}
\int_1^\infty  {\sqrt {{h^{ii}}} } {\left\| \phi _s\wedge\phi _i \right\|_{L_x^\infty }}ds \le {\left\| {du} \right\|_{L_x^2}}{\left\| {{\partial _t}u(t,x)} \right\|_{L_x^2}}.
\end{align}
For $s\in[0,1]$, using (\ref{9.3}) and (\ref{9.2}), we have from H\"older that
\begin{align}\label{VI6}
\int_0^1 {\sqrt {{h^{ii}}} {{\left\| \phi _s\wedge\phi _i\right\|}_{L_x^\infty }}ds}  \le {\left\| {du} \right\|_{L_x^2}}{\left\| {{\partial _t}u(t,x)} \right\|_{L_x^2}}.
\end{align}
Therefore, (\ref{3.30}) follows from (\ref{zfg57xc}), (\ref{VI5}) and (\ref{VI6}).
\end{proof}

\section{The Local and global well-posedness}
If $u(t,x)$ solves (\ref{1}), then it is easy to see
\begin{align}\label{gv6}
\frac{\alpha }{{{\alpha ^2} + {\beta ^2}}}J{u_t} - \frac{\beta }{{{\alpha ^2} + {\beta ^2}}}{u_t} = J\tau (u).
\end{align}
In order to prove the local well-posedness in $\mathcal{H}^3$, we apply the approximating scheme introduced by McGahagan \cite{M}. The novel idea in \cite{M} is that one can use the wave map to approximate the Schr\"odinger map. For any $\delta>0$, we introduce the wave map model equation:
\begin{align}\label{h56}
\left\{ \begin{array}{l}
 \delta {\nabla _t}{\partial _t}u - \alpha \tau (u) + \frac{{\alpha \beta }}{{{\alpha ^2} + {\beta ^2}}}J {{u_t}} + \frac{{{\alpha^2}}}{{{\alpha ^2} + {\beta ^2}}}{u_t} = 0, \\
 u\left( {0,x} \right) = u_0 ,{\partial _t}u\left( {0,x} \right) = g_0^\delta , \\
 \end{array} \right.
\end{align}
where $u(t,x):[0,T]\times\Bbb H^2\to \Bbb H^2$ and $g_0^{\delta}(x)\in T_{u_0(x)}\mathbb{H}^2.$
Formally, we can view (\ref{h56}) as the approximate scheme for (\ref{1}) because of (\ref{gv6}). We remark that different choices of $g^{\delta}_0$ will give the same solution to (\ref{1}), thus it is unnecessary to specify a concrete $g^{\delta}_0$. The local well-posedness of (\ref{1}) is given by the following lemma.
\begin{lemma}\label{local}
Let $\alpha>0$, $\beta\in \Bbb R$. For any initial data $u_0\in \mathcal{H}^3$, there exists $T>0$ depending only on $\|u_0\|_{\mathcal{H}^3}$ such that (\ref{1}) has a unique local solution $u(t,x)\in C([0,T];\mathcal{H}^3)$.
\end{lemma}
\begin{proof}
Rigorously, the volume element of $\Bbb H^2$ should be written as ${\rm{dvol}}_h$, for convenience, we use $dx$ instead in the following proof. And we drop the symbol $\delta$ in $g^{\delta_0}$ sometimes. Fix some $\varepsilon>0$ sufficiently small.
In the coordinates given by (\ref{vg}), (\ref{h56}) can be written as the following semi-linear wave equation
\begin{align}\label{XV12}
&\left( {\delta \frac{{{\partial ^2}{u^k}}}{{\partial {t^2}}} + \overline \Gamma  _{ij}^k\frac{{\partial {u^i}}}{{\partial t}}\frac{{\partial {u^j}}}{{\partial t}} - \alpha {\Delta _{{\Bbb H^2}}}{u^k} - \alpha {h^{ij}}\overline \Gamma  _{mn}^k\frac{{\partial {u^m}}}{{\partial {x^i}}}\frac{{\partial {u^n}}}{{\partial {x^j}}} + \frac{{{\alpha ^2}}}{{{\alpha ^2} + {\beta ^2}}}\frac{{\partial {u^k}}}{{\partial t}}} \right)\frac{\partial }{{\partial {y_k}}} \\
&+ \frac{{\alpha \beta }}{{{\alpha ^2} + {\beta ^2}}}\frac{{\partial {u^i}}}{{\partial t}}J\frac{\partial }{{\partial {y_i}}} = 0.
\end{align}
Using the identity $J\frac{\partial}{\partial y_1}=e^{-y_2}\frac{\partial}{\partial y_2}$, the above equation can be further reduced to a semilinear equation. \\
{\bf Step 1. Local solution for approximate equations.} First notice that $\mathcal{H}^3$ is embedded to $L^{\infty}$  as illustrated in Remark \ref{jiujin}, we can prove the local well-posedness of (\ref{h56}) by the standard contradiction mapping argument. Moreover we can obtain the blow-up criterion: $T>0$ is the lifespan of $(\ref{h56})$ if and only if
$$
\mathop {\lim }\limits_{t \to T} {\left\| {u(t,x)} \right\|_{{\mathcal{H}^3}}} = \infty.
$$
{\bf Step 2. Uniform a-prior estimates for approximate equations.}
We claim that there exists a $T$ depending only on $\|u_0\|_{\mathcal{H}^3}$ such that for all $\delta>0$, $t\in[0,T]$ there exists a a-prior bound
\begin{align}\label{opk}
{\left\| {{u^\delta }(t,x)} \right\|_{{\mathcal{H}^3}}} < C,
\end{align}
for some $C>0$ independent of $\delta$.
Direct calculation gives the following identity which is useful to close the energy estimates
\begin{align}\label{star}
{\partial _t}u = \lambda \left( {{\alpha ^2}\tau (u) - \alpha \beta J\tau (u) - \delta \alpha {\nabla_t}{\partial _t}u + \delta \beta J{\nabla_t}{\partial _t}u} \right),
\end{align}
where $\lambda=\frac{{{\alpha ^2} + \beta {}^2}}{{{\alpha ^3} + {\beta ^2}\alpha }}>0$. And we will frequently use the identity
${\left\langle {JX,X} \right\rangle }=0$ and (\ref{curvature2}). Careful calculations with integration by parts imply for any $X\in u^*(T\Bbb H^2)$, one has
\begin{align}\label{1x}
\int_{{\Bbb H^2}} {\left\langle {\tau (u),X} \right\rangle } dx =  - \int_{{\Bbb H^2}} {{h^{ii}}\left\langle {{\partial _i}u,{\nabla _i}X} \right\rangle } dx.
\end{align}
Define the energy functional by
$$
E(u,\partial_t u)=\frac{\delta}{2}\int_{\Bbb H^2} |\partial_t u|^2dx+\frac{\alpha}{2}\int_{\Bbb H^2} |\partial_i u|^2h^{ii}dx.
$$
Then by (\ref{h56}), we have
$$\frac{d}{{dt}}E(u) =  - \frac{{{\alpha ^2}}}{{{\alpha ^2} + {\beta ^2}}}\left\| {{\partial _t}u} \right\|_{L_x^2}^2.
$$
Thus the energy is decreasing with respect to $t$ and
\begin{align}\label{xx1}
\int_{\Bbb H^2} |\nabla_i u|^2h^{ii}dx\lesssim \delta\|g^{\delta}_0\|^2_{L^2_x}+\|u_0\|^2_{\mathfrak{H}^1}.
\end{align}
Using (\ref{star}), we have
\begin{align}\label{x3}
\frac{1}{2}\frac{d}{{dt}}\left\| {{\partial _t}u} \right\|_{L_x^2}^2 = \lambda {\alpha ^2}\left\langle {{\nabla _t}{\partial _t}u,\tau (u)} \right\rangle  - \lambda \alpha \beta \left\langle {{\nabla _t}{\partial _t}u,J\tau (u)} \right\rangle  - \delta \alpha \lambda \left\langle {{\nabla _t}{\partial _t}u,{\nabla_t}{\partial _t}u} \right\rangle.
\end{align}
We can gain a key negative dominate term from the first term in the right hand side of (\ref{x3}).
In fact, by the non-positiveness of the sectional curvature of the target and expanding $\tau(u)$ to $tr(\nabla du)$, we have
\begin{align*}
\left\langle {{\partial _t}u,{\nabla _t}\tau (u)} \right\rangle  \le \left\langle {{\partial _t}u,{\nabla _i}{\nabla _i}{\partial _t}u} \right\rangle {h^{ii}} - \left\langle {{\partial _t}u,{\nabla _2}{\partial _t}u} \right\rangle.
\end{align*}
Then integration by parts shows
\begin{align}\label{x5}
\lambda {\alpha ^2}\int_0^t {\int_{{\Bbb H^2}} {\left\langle {{\nabla _t}{\partial _t}u,\tau (u)} \right\rangle } dxdt}  \le \lambda {\alpha ^2}\left( \int_{\Bbb H^2} {\left\langle {{\partial _t}u,\tau (u)} \right\rangle |_{t = 0}^{t = t}dx - \int^t_0\int_{{\Bbb H^2}} {\left\langle {{\nabla _i}{\partial _t}u,{\nabla _i}{\partial _t}u} \right\rangle {h^{ii}}dxds} } \right).
\end{align}
Similarly, one has
\begin{align}\label{x6}
\lambda {\alpha |\beta|}\int_0^t {\int_{{\Bbb H^2}} {\left\langle {{\nabla _t}{\partial _t}u,\tau (u)} \right\rangle } dxdt}  \le \lambda {\alpha |\beta|}\left(\int_{\Bbb H^2} {\left\langle {{\partial _t}u,J\tau (u)} \right\rangle |_{t = 0}^{t = t} dx+ C\int^t_0\int_{{\Bbb H^2}} {{{\left| {{\partial _t}u} \right|}^2}{{\left| {{\nabla _i}u} \right|}^2}{h^{ii}}dxds} } \right).
\end{align}
As a consequence of (\ref{x3}), (\ref{x5}) and (\ref{x6}), we obtain from H\"older inequality that
\begin{align}\label{x7}
\left\| {{\partial _t}u} \right\|_{L_x^2}^2 \le 4\lambda {\alpha ^2}{\left\| {{\partial _t}u} \right\|_{L_x^2}}{\left\| {\tau (u)} \right\|_{L_x^2}} + C{\left\| {g_0} \right\|_{L_x^2}}{\left\| {{u_0}} \right\|_{{\mathcal{H}^2}}} + C\int^t_0\left\| {d u} \right\|_{L_x^4}^2\left\| {{\partial _t}u} \right\|_{L_x^4}^2 - 2{\alpha ^2}\lambda \left\| {\nabla {\partial _t}u} \right\|_{L_x^2}^2ds.
\end{align}
Gagliardo-Nirenberg  inequality and diamagnetic inequality imply
$$\left\| {du} \right\|_{L_x^4}^2 \le {\left\| {du} \right\|_{L_x^2}}{\left\| {\nabla du} \right\|_{L_x^2}},\left\| {\partial_t u} \right\|_{L_x^4}^2 \le {\left\| {\nabla {\partial _t}u} \right\|_{L_x^2}}{\left\| {{\partial _t}u} \right\|_{L_x^2}}.$$
Therefore, Young's inequality and (\ref{x7}) yield
$$\left\| {{\partial _t}u} \right\|_{L_x^2}^2 \le C\left\| {\tau (u)} \right\|_{L_x^2}^2 + C{\left\| {g_0} \right\|_{L_x^2}}{\left\| {{u_0}} \right\|_{{\mathfrak{H}^2}}} + C\int_0^t {\left\| {{\partial _t}u} \right\|_{L_x^2}^2 + \left\| {\nabla du} \right\|_{L_x^2}^2}  + \left\| {du} \right\|_{L_x^2}^2ds.
$$
By (\ref{xx1}) and Gronwall inequality, we conclude that
\begin{align}\label{x21}
\left\| {{\partial _t}u} \right\|_{L_x^2}^2 \le C\|g_0\|^2_{L^2_x}+\int^t_0e^{C(t-s)}\|u(s)\|^2_{\mathfrak{H}^2}ds,
\end{align}
where $C$ is independent of $\delta$. Using the non-positiveness of the sectional curvature of the target, we have
\begin{align}\label{x22}
\|\nabla du\|^2_{L^2_x}\lesssim \|\tau(u)\|^2_{L^2_x} +\|du\|^2_{L^2_x}.
\end{align}
Define the second order energy functional by
$$
{E_2}(u,{\partial _t}u) = \frac{\alpha}{2}\int_{{\Bbb H^2}} {|\tau (u){|^2}dx}  + \frac{\delta }{2}\int_{{\Bbb H^2}} {|{\nabla _i}{\partial _t}u{|^2}{h^{ii}}} dx.
$$
Then by (\ref{h56}), we obtain
\begin{align}
&\frac{d}{{dt}}{E_2}(u,{\partial _t}u)\nonumber \\
&= \alpha\int_{{\Bbb H^2}} {\left\langle {{\nabla _t}\tau (u),\tau (u)} \right\rangle dx + } \delta \int_{{\Bbb H^2}} {\left\langle {{\nabla _t}{\nabla _i}{\partial _t}u,{\nabla _i}{\partial _t}u} \right\rangle {h^{ii}}} dx \nonumber\\
&\le \alpha{\int _{{\Bbb H^2}}}\left\langle {{\nabla _t}\tau (u),\tau (u)} \right\rangle dx + \delta \int_{{\Bbb H^2}} {\left\langle {{\nabla _i}{\nabla _t}{\partial _t}u,{\nabla _i}{\partial _t}u} \right\rangle {h^{ii}}} dx + C\delta\int_{{\Bbb H^2}} {{{\left| {{\partial _t}u} \right|}^2}\left| {{\nabla _i}{\partial _t}u} \right|\left| {du} \right|} dx.\label{xx4}
\end{align}
Applying (\ref{h56}) again for the second term in (\ref{xx4}) yields
\begin{align}
&\delta \int_{{\Bbb H^2}} {\left\langle {{\nabla _i}{\nabla _t}{\partial _t}u,{\nabla _i}{\partial _t}u} \right\rangle {h^{ii}}} dx \nonumber\\
&= \alpha {\int _{{\Bbb H^2}}}\left\langle {{\nabla _i}\tau (u),{\nabla _i}{\partial _t}u} \right\rangle {h^{ii}}dx - \frac{{{\alpha ^2}}}{{{\alpha ^2} + {\beta ^2}}}\int_{{\Bbb H^2}} {\left\langle {{\nabla _i}{\partial _t}u,{\nabla _i}{\partial _t}u} \right\rangle {h^{ii}}} dx \nonumber\\
&\le -\alpha {\int _{{\Bbb H^2}}}\left\langle {\tau (u),{\nabla _t}\tau (u)} \right\rangle dx - \frac{{{\alpha ^2}}}{{{\alpha ^2} + {\beta ^2}}}\int_{{\Bbb H^2}} {\left\langle {{\nabla _i}{\partial _t}u,{\nabla _i}{\partial _t}u} \right\rangle {h^{ii}}} dx \nonumber\\
&+ C\int_{{\Bbb H^2}} {\left| {{\partial _t}u} \right|\left| {\tau (u)} \right|{{\left| {du} \right|}^2}} dx + C\int_{{\Bbb H^2}} {\left| {\nabla {\partial _t}u} \right|\left| {\tau (u)} \right|} dx.\label{xx7}
\end{align}
Thus we have from (\ref{xx4}), (\ref{xx7}) that
\begin{align*}
\frac{d}{{dt}}{E_2}(u,{\partial _t}u) &\le  - \frac{{{\alpha ^2}}}{{{\alpha ^2} + {\beta ^2}}}\int_{{\Bbb H^2}} {\left\langle {{\nabla _i}{\partial _t}u,{\nabla _i}{\partial _t}u} \right\rangle {h^{ii}}} dx + C\delta\int_{{\Bbb H^2}} {{{\left| {{\partial _t}u} \right|}^2}\left| {{\nabla}{\partial _t}u} \right|\left| {du} \right|} dx \\
&+ C\int_{{\Bbb H^2}} {\left| {{\partial _t}u} \right|\left| {\tau (u)} \right|{{\left| {du} \right|}^2}} dx + C\int_{{\Bbb H^2}} {\left| {\nabla {\partial _t}u} \right|\left| {\tau (u)} \right|} dx.
\end{align*}
Then Young's inequality and Gagliardo-Nirenberg  inequality show
\begin{align}
\frac{d}{{dt}}{E_2}(u,{\partial _t}u) &\le \int_{{\Bbb H^2}} {\left( {{\delta ^4}{{\left| {{\partial _t}u} \right|}^8} + {{\left| {du} \right|}^4} + {{\left| {\tau (u)} \right|}^2}} \right)} dx + \left( {\int_{{\Bbb H^2}} {{{\left| {{\partial _t}u} \right|}^2}dx} } \right)\left( {\int_{{\Bbb H^2}} {{{\left| {du} \right|}^8}dx} } \right) \nonumber\\
&\lesssim {E_2}{(u,{\partial _t}u)^4} + \left\| u \right\|_{{\mathfrak{H}^2}}^4 + \left\| u \right\|_{{\mathfrak{H}^2}}^2 + \left\| u \right\|_{{\mathfrak{H}^2}}^8\left\| {{\partial _t}u} \right\|_{L_x^2}^2.\label{xxx1}
\end{align}
By (\ref{x22}), (\ref{xx1}), we see that
\begin{align}\label{xxx4}
\left\| u \right\|^2_{{\mathfrak{H}^2}}\lesssim {E_2}(u,{\partial _t}u)+E(u_0).
\end{align}
Hence (\ref{x22}), (\ref{xxx1}) give
$${E_2}(u,{\partial _t}u) \le \int_0^t {F\left( {{E_2}(u,{\partial _t}u) + E({u_0})} \right)ds}  + {E_2}({u_0},{g_0}),
$$
where $F:\Bbb R\to \Bbb R^+$ is some $C^2$ function.
Then we have from Gronwall inequality that there exists $T>0$ such that for all $t\in[0,T]$
$${E_2}(u,{\partial _t}u)\le C(T),$$
which combined with (\ref{xxx4}) yields
\begin{align*}
\|u\|^2_{\mathfrak{H}^2}\lesssim C(T).
\end{align*}
Define the third order energy functional by
\begin{align}
E_3(u,\partial_t u)=\frac{\alpha}{2}\int_{\Bbb H^2}|\nabla \tau(u)|^2dx+\frac{\delta}{2}\int_{\Bbb H^2} |\nabla^2 \partial_t u|^2dx.
\end{align}
Before calculating the differentiation with respect to $t$ of $E_3$, we first point out a useful inequality which can be verified by integration by parts
\begin{align}\label{V6}
\|\nabla^2 du\|^2_{L^2_x}\lesssim \|\nabla \tau(u)\|^2_{L^2_x}+ \|du\|^6_{L^6_x}+\|\nabla d u\|^2_{L^4_x}\|d u\|^2_{L^4_x}+\|u \|^3_{\mathfrak{H}^2}.
\end{align}
Thus (\ref{V6}), Gagliardo-Nirenberg  inequality and Young inequality further yield
\begin{align}\label{equil}
\|\nabla^2 du\|^2_{L^2_x}\lesssim \mathcal{P}(\|u \|^2_{\mathfrak{H}^2})+\|\nabla \tau(u)\|^2_{L^2_x},
\end{align}
where $\mathcal{P}(x)$ is some polynomial.
Integration by parts and (\ref{h56}) give
\begin{align}
 &\frac{d}{{dt}}{E_3}(u,{\partial _t}u) \\
 &\le  - \frac{{{\alpha ^2}}}{{{\alpha ^2} + {\beta ^2}}}{\int_{{\Bbb H^2}} {\left| {{\nabla ^2}{\partial _t}u} \right|} ^2}dx + \delta \int_{{\Bbb H^2}} {{{\left| {{\partial _t}u} \right|}^2}{{\left| {du} \right|}^2}\left| {{\nabla ^2}{\partial _t}u} \right|} dx + \delta \int_{{\Bbb H^2}} {{{\left| {{\partial _t}u} \right|}^2}\left| {du} \right|\left| {{\nabla ^2}{\partial _t}u} \right|} dx \nonumber\\
 &+ \delta \int_{{\Bbb H^2}} {\left| {{\partial _t}u} \right|\left| {du} \right|\left| {{\nabla ^2}{\partial _t}u} \right|} \left| {\nabla {\partial _t}u} \right|dx + C\int_{{\Bbb H^2}} {\left| {{\partial _t}u} \right|\left| {du} \right|\left| {{\nabla ^2}u} \right|} \left| {\nabla \tau (u)} \right|dx \nonumber\\
 &+ C\int_{{\Bbb H^2}} {\left| {{\partial _t}u} \right|{{\left| {du} \right|}^3}} \left| {\nabla \tau (u)} \right|dx + C\int_{{\Bbb H^2}} {\left| {\nabla {\partial _t}u} \right|{{\left| {du} \right|}^2}} \left| {\nabla \tau (u)} \right|dx \nonumber\\
 &+ C\int_{{\Bbb H^2}} {\left| {{\partial _t}u} \right|\left| {du} \right|} \left| {\nabla \tau (u)} \right|\left| {\tau (u)} \right|dx.\label{XV1}
\end{align}
By Gagliardo-Nirenberg  inequality and Young inequality, we have
\begin{align*}
\delta \int_{{\Bbb H^2}} {{{\left| {{\partial _t}u} \right|}^2}{{\left| {du} \right|}^2}\left| {{\nabla ^2}{\partial _t}u} \right|} dx &\le {E_3}(u,{\partial _t}u) + \int_{{\Bbb H^2}} {{{\left| {{\partial _t}u} \right|}^4}{{\left| {du} \right|}^4}dx}  \\
&\le {E_3}(u,{\partial _t}u) + \left\| {du} \right\|_{L_x^\infty }^8\left\| {{\partial _t}u} \right\|_{L_x^2}^6 + \varepsilon \int_{{\Bbb H^2}} {{{\left| {{\nabla ^2}{\partial _t}u} \right|}^2}dx}.
\end{align*}
Similarly, we have
\begin{align*}
 \delta \int_{{\Bbb H^2}} {{{\left| {{\partial _t}u} \right|}^2}\left| {du} \right|\left| {{\nabla ^2}{\partial _t}u} \right|} dx &\le {E_3}(u,{\partial _t}u) + \left\| {du} \right\|_{L_x^\infty }^4\left\| {{\partial _t}u} \right\|_{L_x^2}^6 + \varepsilon \int_{{\Bbb H^2}} {{{\left| {{\nabla ^2}{\partial _t}u} \right|}^2}dx} \\
 \delta \int_{{\Bbb H^2}} {\left| {{\partial _t}u} \right|\left| {du} \right|\left| {{\nabla ^2}{\partial _t}u} \right|} \left| {\nabla {\partial _t}u} \right|dx &\le {E_3}(u,{\partial _t}u) + \varepsilon \int_{{\Bbb H^2}} {{{\left| {{\nabla ^2}{\partial _t}u} \right|}^2}dx}  + \left\| {du} \right\|_{L_x^\infty }^8\left\| {{\partial _t}u} \right\|_{L_x^2}^6 \\
 \int_{{\Bbb H^2}} {\left| {{\partial _t}u} \right|\left| {du} \right|\left| {{\nabla ^2}u} \right|} \left| {\nabla \tau (u)} \right|dx &\le {E_3}(u,{\partial _t}u) + \varepsilon \int_{{\Bbb H^2}} {{{\left| {{\nabla ^2}{\partial _t}u} \right|}^2}dx}  + \left\| {du} \right\|_{L_x^\infty }^8\left\| {{\partial _t}u} \right\|_{L_x^2}^6 + \left\| u \right\|_{{\mathfrak{H}^3}}^4 \\
 \int_{{\Bbb H^2}} {\left| {\nabla {\partial _t}u} \right|{{\left| {du} \right|}^2}} \left| {\nabla \tau (u)} \right|dx &\le {E_3}(u,{\partial _t}u) + \varepsilon \int_{{\Bbb H^2}} {{{\left| {{\nabla ^2}{\partial _t}u} \right|}^2}dx}  + \left\| {du} \right\|_{L_x^4}^4\left\| {{\partial _t}u} \right\|_{L_x^2}^2 \\
 \int_{{\Bbb H^2}} {\left| {{\partial _t}u} \right|\left| {du} \right|} \left| {\nabla \tau (u)} \right|\left| {\tau (u)} \right|dx &\le {E_3}(u,{\partial _t}u) + \varepsilon \int_{{\Bbb H^2}} {{{\left| {{\nabla ^2}{\partial _t}u} \right|}^2}dx}  + \left\| u \right\|_{{\mathfrak{H}^3}}^4.
 \end{align*}
Therefore,we conclude from (\ref{XV1}) and (\ref{equil})  that
$${E_3}(u,{\partial _t}u) \le \int_0^t {G\left( {{E_3}(u,{\partial _t}u) + {E_2}(u,{\partial _t}u) + E({u_0},{g_0})} \right)ds}  + {E_3}({u_0},{g_0}) + {E_2}({u_0},{g_0}) + E({u_0},{g_0}),
$$
where $G:\Bbb R\to\Bbb R^+$ is a $C^2$ function. Thus by Gronwall inequality, we have there exists $T_1>0$ such that for any $\delta>0$, $t\in[0,T_1]$ such that (\ref{h56}) has a local solution in $L^{\infty}([0,T_1];\mathfrak{H}^3)$ and
$$
\|u^{\delta}\|_{\mathfrak{H}^3}\le C(T_1).
$$
By Lemma \ref{new}, this implies the uniform a-prior bound for $\|u^{\delta}\|_{\mathcal{H}^3}$. Now letting $\delta\to 0$, we have a sequence of solutions $u^{\delta_n}$ of (\ref{h56}), which converge to $u\in L^{\infty}([0,T];\mathcal{H}^3)$ in the $weak^*$ sense.
The limit map $u$ satisfies (\ref{1}) a.e. on $[0,T_1]\times \Bbb H^2$. The proof for this fact is standard as shown in \cite{M}. Now we prove the regularity of $u$ with respect to $t$. Integrating (\ref{XV1}) with respect to $t$ in $[0,T]$ yields $\partial_tu\in L^2([0,T];\mathfrak{H}^2)$ and $u\in L^2([0,T];\mathcal{H}^4)$. Thus by reducing these bounds to $(u^1,u^2)$, the standard argument implies $(u^1,u^2)\in C([0,T];H^3(\Bbb H^2;\Bbb R^2))$. Thus we conclude $u\in C([0,T];\mathcal{H}^3)$. The proof of the uniqueness is highly standard by \cite{M}.
\end{proof}

The global well-posedness is given by the following proposition.
\begin{proposition}\label{global}
Let $\alpha>0$, $\beta\in \Bbb R$. For any initial data $u_0\in \mathcal{H}^3$, there exists a unique global solution $u\in C([0,T],\mathcal{H}^3)$ of (\ref{1}) for all $T>0$. Furthermore, we have $\|u\|_{\mathfrak{H}^2}\lesssim C(u_0)$.
\end{proposition}
\begin{proof} By the local well-posedness in Lemma \ref{local}, it suffices to obtain a uniform bound for $\|u\|_{\mathcal{H}^3}$ with respect to $t\in[0,T]$. As before we introduce three energy functionals:
\begin{align*}
{\mathcal{E}_1}(u) = \frac{1}{2}\int_{{\Bbb H^2}} {{{\left| {\nabla u} \right|}^2}} dx,\mbox{  }{\mathcal{E}_2}(u) = \frac{1}{2}\int_{{\Bbb H^2}} {{{\left| {{\partial _t}u} \right|}^2}} {\rm{dvol_h}},\mbox{  }{\mathcal{E}_3}(u) = \frac{1}{2}\int_{{\Bbb H^2}} {{{\left| {\nabla {\partial _t}u} \right|}^2}} {\rm{dvol_h}}.
\end{align*}
By integration by parts and (\ref{1}), we have
\begin{align*}
\frac{d}{{dt}}{\mathcal{E}_1}(u) =  - \alpha \int_{{\Bbb H^2}} {{{\left| {\tau (u)} \right|}^2}} {\rm{dvol_h}}.
\end{align*}
Thus the energy is decreasing with respect to $t$ and
\begin{align}\label{V1}
\|du\|^2_{L^2_x}+\alpha\int^t_0\|\partial_t u\|^2_{L^2_x}\le \mathcal{E}_1(u_0).
\end{align}
Since $u$ satisfies (\ref{1}), then
\begin{align}
 \int_{{\Bbb H^2}} {{{\left| {{\partial _t}u} \right|}^2}} dx &= \left( {{\alpha ^2} + {\beta ^2}} \right)\int_{{\Bbb H^2}} {{{\left| {\tau (u)} \right|}^2}} dx \label{V2}\\
 \int_{{\Bbb H^2}} {{{\left| {\nabla {\partial _t}u} \right|}^2}} dx &= \left( {{\alpha ^2} + {\beta ^2}} \right)\int_{{\Bbb H^2}} {{{\left| {\nabla \tau (u)} \right|}^2}} dx \label{V3}
\end{align}
By (\ref{1}) and integration by parts, one has
\begin{align}
 &\frac{d}{{dt}}{\mathcal{E}_2}(u) = \int_{{\Bbb H^2}} {\left\langle {{\nabla _t}{\partial _t}u,{\partial _t}u} \right\rangle } dx \nonumber\\
 &= \alpha \int_{{\Bbb H^2}} {\left\langle {{\nabla _t}\tau (u),{\partial _t}u} \right\rangle } dx + \beta \int_{{\Bbb H^2}} {\left\langle {J{\nabla _t}\tau (u),{\partial _t}u} \right\rangle }  \nonumber\\
 &\le  - \alpha \int_{{\Bbb H^2}} {\left\langle {\nabla {\partial _t}(u),\nabla {\partial _t}u} \right\rangle } dx + C\int_{{\Bbb H^2}} {{{\left| {du} \right|}^2}{{\left| {{\partial _t}u} \right|}^2}dx}.\label{j089}
\end{align}
H\"older and Gagliardo-Nirenberg inequality imply for some fixed sufficiently small $\varepsilon>0$
$$\int_{{\Bbb H^2}} {{{\left| {du} \right|}^2}{{\left| {{\partial _t}u} \right|}^2}dx}  \le {\left\| {\nabla du} \right\|_{L_x^2}}{\left\| {du} \right\|_{L_x^2}}{\left\| {\nabla {\partial _t}u} \right\|_{L_x^2}}{\left\| {{\partial _t}u} \right\|_{L_x^2}} \le {\mathcal{E}_1}\left( {{u_0}} \right)\left\| {\nabla du} \right\|_{L_x^2}^2\left\| {{\partial _t}u} \right\|_{L_x^2}^2 + \varepsilon \left\| {\nabla {\partial _t}u} \right\|_{L_x^2}^2.
$$
Therefore, we deduce from (\ref{j089}) that
$$\frac{d}{{dt}}{\mathcal{E}_2}(u) \le {\mathcal{E}_1}\left( {{u_0}} \right)\left\| {\nabla du} \right\|_{L_x^2}^2\left\| {{\partial _t}u} \right\|_{L_x^2}^2.
$$
Then applying (\ref{x22}), (\ref{V2}), we have
\begin{align}
\frac{d}{{dt}}{\mathcal{E}_2}(u) \le {\mathcal{E}_1}\left( {{u_0}} \right)\left( {{\mathcal{E}_2}(u) + \mathcal{E}_1({u_0})} \right)\left\| {{\partial _t}u} \right\|_{L_x^2}^2.
\end{align}
Gronwall inequality shows
$${\mathcal{E}_2}(u(t)) + {\mathcal{E}_1}\left( {{u_0}} \right) \le \left( {\left\| {{u_0}} \right\|_{{H^2}}^2 + {\mathcal{E}_1}\left( {{u_0}} \right)} \right){e^{C{\mathcal{E}_1}\left( {{u_0}} \right)\int_0^t {\left\| {{\partial _t}u} \right\|_{L_x^2}^2d\tau } }}.
$$
Thus by (\ref{V1}), we obtain
$${\mathcal{E}_2}(u(t))\le C(\|u_0\|_{\mathfrak{H}^2}).$$
Again using (\ref{x22}), (\ref{V2}), we conclude
\begin{align}\label{V5}
\|u(t)\|_{\mathfrak{H}^2}+\alpha\int^t_0\|\nabla\partial_t u\|^2_{L^2_x}ds\le C(\|u_0\|_{\mathfrak{H}^2}).
\end{align}
Integration by parts and (\ref{1}) yield
\begin{align*}
 \frac{d}{{dt}}{\mathcal{E}_3}(u(t)) &\le  - \alpha \int_{{\Bbb H^2}} \big({{{\left| {{\nabla ^2}{\partial _t}u} \right|}^2}}  + C\left| {\nabla u} \right|\left| {\nabla {\partial _t}u} \right|{\left| {{\partial _t}u} \right|^2} + C\left| {\nabla u} \right|\left| {\nabla {\partial _t}u} \right|{\left| {{\partial _t}u} \right|^2} \big)dx\\
 &+ C\int_{\Bbb H^2}\big({\left| {\nabla u} \right|^3}\left| {{\partial _t}u} \right|\left| {\nabla {\partial _t}u} \right| + C\left| {{\nabla ^2}u} \right|\left| {\nabla u} \right|\left| {{\partial _t}u} \right|\left| {\nabla {\partial _t}u} \right| + C{\left| {{\partial _t}u} \right|^2}{\left| {\nabla u} \right|^4}\big)dx \\
 &+ C\int_{\Bbb H^2}\big({\left| {\nabla {\partial _t}u} \right|^2}{\left| {\nabla u} \right|^2} + C{\left| {\nabla u} \right|^2}\left| {{\nabla ^2}{\partial _t}u} \right|\left| {{\partial _t}u} \right|\big)dx.
\end{align*}
By Gagliardo-Nirenberg  inequality and Young inequality, we see
\begin{align*}
 \int_{{\Bbb H^2}} {\left| {\nabla u} \right|\left| {\nabla {\partial _t}u} \right|{{\left| {{\partial _t}u} \right|}^2}dx}&\le \left\| {{\partial _t}u} \right\|_{L_x^2}{\left\| {du} \right\|_{L_x^2}^{\frac{1}{2}}}\left\| {\nabla {\partial _t}u} \right\|_{L_x^2}^{3/2}\left\| {{\nabla ^2}{\partial _t}u} \right\|_{L_x^2}^{\frac{1}{2}}\left\| {\nabla du} \right\|_{L_x^2}^{\frac{1}{2}} \\
 &\le C({u_0})\left\| {\nabla {\partial _t}u} \right\|_{L_x^2}^2 + \varepsilon \left\| {{\nabla ^2}{\partial _t}u} \right\|_{L_x^2}^2.
\end{align*}
Similarly we have
\begin{align*}
 \int_{{\Bbb H^2}} {{{\left| {du} \right|}^3}\left| {{\partial _t}u} \right|\left| {\nabla {\partial _t}u} \right|dx}  &\le \left\| {{\nabla ^2}{\partial _t}u} \right\|_{L_x^2}^{\frac{1}{2}}\left\| {\nabla {\partial _t}u} \right\|_{L_x^2}^{\frac{1}{2}}{\left\| {{\partial _t}u} \right\|_{L_x^2}}\left\| {du} \right\|_{L_x^{12}}^3 \\
 &\le C({u_0})\left\| {\nabla {\partial _t}u} \right\|_{L_x^2}^2 + \varepsilon \left\| {{\nabla ^2}{\partial _t}u} \right\|_{L_x^2}^2 \\
 \int_{{\Bbb H^2}} {\left| {\nabla du} \right|\left| {du} \right|\left| {{\partial _t}u} \right|\left| {\nabla {\partial _t}u} \right|} dx &\le \left\| {{\nabla ^2}{\partial _t}u} \right\|_{L_x^2}^{\frac{1}{2}}{\left\| {\nabla {\partial _t}u} \right\|_{L_x^2}}\left\| {{\nabla ^2}du} \right\|_{L_x^2}^{\frac{1}{2}}\left\| {\nabla du} \right\|_{L_x^2}^{\frac{1}{2}}{\left\| {du} \right\|_{L_x^4}}\left\| {{\partial _t}u} \right\|_{L_x^2}^{\frac{1}{2}} \\
 &\le C({u_0})\left\| {\nabla {\partial _t}u} \right\|_{L_x^2}^2\left\| {{\nabla ^2}du} \right\|_{L_x^2}^2 + \varepsilon \left\| {{\nabla ^2}{\partial _t}u} \right\|_{L_x^2}^2.
 \end{align*}
 The remaining three terms are easier to bound:
 \begin{align*}
 \int_{{\Bbb H^2}} {{{\left| {{\partial _t}u} \right|}^2}{{\left| {\nabla u} \right|}^4}} dx &\le C({u_0})\left\| {\nabla {\partial _t}u} \right\|_{L_x^2}^2 \\
 \int_{{\Bbb H^2}} {{\left| {\nabla {\partial _t}u} \right|}^2}{{\left| {\nabla u} \right|}^2}dx &\le \varepsilon \left\| {{\nabla ^2}{\partial _t}u} \right\|_{L_x^2}^2 + C({u_0}) \\
 \int_{{\Bbb H^2}} {{{\left| {\nabla u} \right|}^2}\left| {{\nabla ^2}{\partial _t}u} \right|\left| {{\partial _t}u} \right|} dx &\le \varepsilon \left\| {{\nabla ^2}{\partial _t}u} \right\|_{L_x^2}^2 + C({u_0})\left\| {\nabla {\partial _t}u} \right\|_{L_x^2}^2.
 \end{align*}
 Thus we conclude that
 $$
\frac{d}{{dt}}\left\| {\nabla {\partial _t}u} \right\|_{L_x^2}^2 \lesssim C({u_0})\left( {1 + \left\| {{\nabla ^2}du} \right\|_{L_x^2}^2} \right)\left\| {\nabla {\partial _t}u} \right\|_{L_x^2}^2.
$$
Then Gronwall inequality, (\ref{V5}) and (\ref{V6}) imply
\begin{align}
\left\| {\nabla {\partial _t}u(t)} \right\|_{L_x^2}^2 &\le \left\| {\nabla {\partial _t}u(0)} \right\|_{L_x^2}^2{e^{C({u_0})\left( {\int_0^t {(1 + \left\| {{\nabla ^2}du} \right\|_{L_x^2}^2)ds} } \right)}} \nonumber\\
&\lesssim \left\| {\nabla {\partial _t}u(0)} \right\|_{L_x^2}^2{e^{C({u_0})t}}.\label{keypoint}
\end{align}
Therefore we obtain again from (\ref{V5}) and (\ref{V6}) that
$$\|u(t)\|_{\mathfrak{H}^3}\le C({u_0},{e^{Ct}}).
$$
By Lemma \ref{new}, we conclude $\|u(t)\|_{\mathcal{H}^3}\le Y(t)$ for some continuous function $Y:\Bbb R\to \Bbb R^+$. By the local well-posedness, $u(t)$ exists globally in $[0,\infty)\times \Bbb H^2$.
\end{proof}

Proposition \ref{global} has an important corollary which reveals the nonlinear smoothing effect for the heat tension field $\tau(u)$ due to the dissipative nature of $(\ref{1}).$ Indeed, (\ref{V5}) shows
\begin{corollary}\label{4.1a}
Let $\alpha>0$, $\beta\in \Bbb R$, and $u$ be a solution to (\ref{1}) in $C(\Bbb R^+;\mathcal H^3)$, then we have
\begin{align}
\int_0^\infty  {\left\| {\nabla \partial_t u} \right\|_{L_x^2}^2} dt+\int_0^\infty  \|\partial_tu\|^2_{L^2_x}dt \le C({\left\| {{u_0}} \right\|_{{\mathfrak{H}^2}}}).
\end{align}
\end{corollary}

Similar proof as Proposition \ref{global} can give a local bound for $\|\nabla \partial_t\tilde{u}\|_{L^2_x}$ with respect to $s$.
\begin{lemma}
Let $\tilde{u}$ be the solution to (\ref{8.29.2}) with initial data $u(t,x)$, then we have
\begin{align}\label{ttu}
\|\nabla \partial_t\tilde{u}\|_{L^2_x}\lesssim e^{sC}\|\partial_t u\|_{L^2_x}.
\end{align}
\end{lemma}

\section{Convergence to a harmonic map}
In this section, we always fix the frame $\Xi$ in Proposition \ref{3.3} to be $\Xi(x)=\Theta(Q_{\infty}(x))$.
\subsection{Estimates of the heat tension field}
Recall that the heat tension filed $\phi_s$ satisfies
\begin{align}\label{heat}
\phi_s= h^{ij}D_i\phi_j-h^{ij}\Gamma^k_{ij}\phi_k.
\end{align}
And we define the LL tension filed by
\begin{align}\label{LL}
w\triangleq \phi_t- z\big(h^{ij}D_i\phi_j-h^{ij}\Gamma^k_{ij}\phi_k\big).
\end{align}
In fact (\ref{heat}) is the gauged equation for the heat flow equation, and (\ref{LL}) is the gauged equation for the LL flow (\ref{1}), see Lemma 3.1.
The evolution of $\phi_s$ and $w$ with respect to $t$ and $s$ respectively is given by the following lemma.

\begin{lemma}\label{asdf}
The tension fields $\phi_i, i=1,2$, $\phi_s$, $w$ satisfy
\begin{align}
D_t\phi_s&=zh^{ij}D_iD_j\phi_s-zh^{ij}\Gamma^k_{ij}D_k\phi_s+\partial_s w-zh^{ij}(\phi_i \wedge\phi_s)\phi_j \label{heating}\\
\partial_s w&=h^{ij}D_iD_jw-h^{ij}\Gamma^k_{ij}D_kw+h^{ij}(\phi_t\wedge \phi_i)\phi_j-zh^{ij}(\phi_s \wedge\phi_i)\phi_j.\label{LLing}
\end{align}
\end{lemma}
\begin{proof}
By the torsion free identity and the commutator identity, we have
\begin{align*}
D_t\phi_s&=D_s\phi_t=D_s(w+zh^{ij}D_i\phi_j-zh^{ij}\Gamma^k_{ij}\phi_k)\\
&=\partial_s w+zh^{ij}D_sD_i\phi_j-zh^{ij}\Gamma^k_{ij}D_s\phi_k\\
&=\partial_s w+zh^{ij}D_iD_j\phi_s-zh^{ij}\Gamma^k_{ij}D_k\phi_s+zh^{ij}(\phi_s \wedge\phi_i)\phi_j.
\end{align*}
Thus (\ref{heating}) is verified. The rest is to prove (\ref{LLing}). By (\ref{LL}), the torsion free identity and the commutator identity (\ref{commut}), we obtain
\begin{align*}
\partial_sw&=D_sw=D_s(\phi_t-zh^{ij}D_i\phi_j-zh^{ij}\Gamma^k_{ij}\phi_k)\\
&=D_s \phi_t-zh^{ij}D_sD_i\phi_j+zh^{ij}\Gamma^k_{ij}D_s\phi_k\\
&=D_s \phi_t-zh^{ij}D_iD_j\phi_s+zh^{ij}\Gamma^k_{ij}D_k\phi_s+zh^{ij}(\phi_s\wedge\phi_i)\phi_j\\
&=D_s \phi_t-h^{ij}D_iD_j(\phi_t-w)+zh^{ij}\Gamma^k_{ij}D_k(\phi_t-w)+zh^{ij}(\phi_s\wedge\phi_i)\phi_j\\
&=h^{ij}D_iD_jw-h^{ij}\Gamma^k_{ij}D_kw+D_s \phi_t-h^{ij}D_iD_j\phi_t+h^{ij}\Gamma^k_{ij}D_k\phi_t+zh^{ij}(\phi_s\wedge\phi_i)\phi_j.
\end{align*}
The torsion free identity and the commutator identity further show
\begin{align*}
&D_s \phi_t-h^{ij}D_iD_j\phi_t+h^{ij}\Gamma^k_{ij}D_k\phi_t\\
&=D_t \phi_s-h^{ij}D_iD_t\phi_j+h^{ij}\Gamma^k_{ij}D_t\phi_k\\
&=D_t \phi_s-h^{ij}D_tD_i\phi_j+h^{ij}\Gamma^k_{ij}D_t\phi_k-h^{ij}(\phi_i\wedge\phi_t)\phi_j\\
&=D_t \phi_s-D_t(h^{ij}D_i\phi_j-h^{ij}\Gamma^k_{ij}\phi_k)-h^{ij}(\phi_i\wedge\phi_t)\phi_j\\
&=D_t \phi_s-D_t\phi_s-h^{ij}(\phi_i\wedge\phi_t)\phi_j.
\end{align*}
Combining the two above equalities together yields (\ref{LLing}).
\end{proof}

\begin{lemma}\label{hushuo}
The heat tension filed $\phi_s$ satisfies
\begin{align}
\|\nabla \phi_s(s,t,x)\|_{L^2_x}&\lesssim e^{Cs}\|\nabla\partial_tu\|_{L^2_x}+\|\partial_tu\|_{L^2_x}\label{XVV1}\\
\|\phi_s(s,t,x)\|_{L^2_x}&\lesssim \|\partial_tu\|_{L^2_x}.\label{XVV2}
\end{align}
The LL filed $w$ satisfies
\begin{align}
\|w(s,t,x)\|_{L^2_x}&\lesssim \|\partial_tu\|_{L^2_x} \label{XVV5}\\
\|\nabla w(s,t,x)\|_{L^2_x}&\lesssim s\|\nabla\partial_tu\|_{L^2_x} +\|\nabla\partial_tu\|_{L^2_x}\label{XVV3}\\
\|\partial_s w(s,t,x)\|_{\dot{H}^{-1}_x}&\lesssim \big(e^{Cs} + 1 + c(s)\big) \|\partial_tu\|_{L^2_x}.\label{XVV4}
\end{align}
\end{lemma}
\begin{proof}
(\ref{XVV2}) follows by (\ref{VI3}) and the $L^2$ to $L^2$ boundedness of $e^{s\Delta_{\Bbb H^2}}$. (\ref{XVV5}) is a consequence of the identity $w=\phi_t-z\phi_s$ and (\ref{VI9}).
Let $\beta=0$ in Proposition \ref{global}, since $\widetilde{u}$ satisfies the heat flow equation, (\ref{keypoint}) implies
$$
\|\nabla\partial_s\widetilde{u}(s,t,x)\|_{L^2_x}\lesssim e^{Cs}\|\nabla\partial_s\widetilde{u}(0,t,x)\|_{L^2_x}.
$$
Thus the heat flow equation and (\ref{1}) shows
\begin{align}\label{090}
\|\nabla\partial_s\widetilde{u}(s,t,x)\|_{L^2_x}\lesssim e^{Cs}\|\nabla\partial_tu(t,x)\|_{L^2_x}.
\end{align}
By the definition  $\nabla {\psi^j_s} = {\partial _i}\left\langle {{\partial _s}{\tilde{u}},{e_j}} \right\rangle d{x^i}$ and the comparability
\begin{align}\label{ttu3}
{\partial _i}\left\langle {{\partial _s}\widetilde{u},{e_i}} \right\rangle  = \left\langle {{\nabla _i}{\partial _s}\widetilde{u},{e_i}} \right\rangle  + \left\langle {{\partial _s}\widetilde{u},{\nabla _i}{e_i}} \right\rangle,
\end{align}
we obtain
\begin{align*}
 {\left| {\nabla {\phi _s}} \right|^2} \le {h^{ii}}{\left| {{\nabla _i}{\partial _s}\widetilde{u}} \right|^2} + {h^{ii}}{\left| {{\partial _s}\widetilde{u}} \right|^2}{\left| {{A_i}} \right|^2}\le {\left| {\nabla {\partial _s}\widetilde{u}} \right|^2} + {\left| {{\partial _s}\widetilde{u}} \right|^2}{h^{ii}}{\left| {{A_i}} \right|^2}.
\end{align*}
Then (\ref{3.30}) and $\|u(t)\|_{\mathfrak{H}^2}\le C(u_0)$ proved in Proposition \ref{global} yield
$${\left\| {\nabla {\phi _s}} \right\|_{L_x^2}} \le {\left\| {\nabla {\partial _s}\widetilde{u}} \right\|_{L_x^2}} + C({u_0}){\left\| {{\partial _s}\widetilde{u}} \right\|_{L_x^2}},
$$
which combined with (\ref{090}) gives
$${\left\| {\nabla {\phi _s}} \right\|_{L_x^2}} \le {e^{Cs}}{\left\| {\nabla {\partial _t}u} \right\|_{L_x^2}} + C({u_0}){\left\| {{\partial _s}u} \right\|_{L_x^2}}.
$$
Thus (\ref{XVV1}) follows from (\ref{XVV2}).
By the definition $w=\phi_t-z\phi_s$, in order to verify (\ref{XVV3}), it suffices to prove
\begin{align}\label{ttu2}
\|\nabla \phi_t\|_{L^2}\lesssim e^{Cs}\|\partial_t u\|_{L^2}.
\end{align}
Similar calculations as (\ref{ttu3}) imply
$${\left\| {\nabla {\phi _t}} \right\|_{L_x^2}} \le {\left\| {\nabla {\partial _t}\widetilde{u}} \right\|_{L_x^2}} + {\left\| {{\partial _t}\widetilde{u}} \right\|_{L_x^2}}.
$$
Hence (\ref{XVV3}) follows by (\ref{ttu}) and (\ref{VI4}). It remains to prove (\ref{XVV4}).
Expanding $D_i$ in (\ref{LLing}) as $\partial_i+\sqrt{-1}A_i$ yields
\begin{align*}
{\partial _s}{w} = {\Delta _{{\Bbb H^2}}}{w} + \sqrt{-1}{h^{ii}}{\partial _i}{A_i}{w} +2\sqrt{-1} {h^{ii}}{A_i}{\partial _i}{w} -{h^{ii}}{A_i}{A_i}{w} + {h^{ii}}( \phi_t\wedge\phi_i){\phi _j}-z{h^{ii}}(\phi_s\wedge\phi_i){\phi_i}.
\end{align*}
Then we have from Lemma \ref{wusijue} that
\begin{align}\label{k908}
\|{\partial _s}{w}\|_{\dot{H}^{-1}_x}&\lesssim \|\nabla w\|_{L^2_x}+\|{h^{ii}}{\partial _i}{A_i}{w}\|_{\dot{H}^{-1}}+\|{h^{ii}}{A_i}{\partial _i}{w}\|_{L^2_x}+\|{h^{ii}}{A_i}A_i{w}\|_{L^2_x}+\|{h^{ii}}(\phi_t\wedge\phi_i){\phi _i}\|_{L^2_x}\nonumber\\
&+\|{h^{ii}}(\phi_s\wedge\phi_i){\phi _i}\|_{L^2_x}.
\end{align}
The later three terms on the right hand side of (\ref{k908}) are easy to bound by H\"older and (\ref{3.30}):
\begin{align}
&{\left\| {{h^{ii}}{A_i}{\partial _i}w} \right\|_{L_x^2}} + {\left\| {{h^{ii}}{A_i}{A_i}w} \right\|_{L_x^2}} + {\left\| {h^{ii}}( {\phi_t}\wedge\phi_i){\phi _i} \right\|_{L_x^2}} + {\left\| {{h^{ii}}( \phi _s\wedge\phi _i){\phi _i}} \right\|_{L_x^2}}\nonumber\\
&\lesssim {\left\| {\nabla w} \right\|_{L_x^2}} + {\left\| w \right\|_{L_x^2}} + {\left\| {{\partial _t}\widetilde{u}} \right\|_{L_x^2}}{\left\| \widetilde{e } \right\|_{L_x^\infty }} + {\left\| {{\partial _s}\widetilde{u}} \right\|_{L_x^2}}{\left\| \widetilde{e} \right\|_{L_x^\infty }}.\label{hvnk}
\end{align}
Furthermore,  by (\ref{XVV5}), (\ref{3.14a}), we have the following acceptable bound for (\ref{hvnk})
$$
(\ref{hvnk})\le \left( {{e^{Cs}} + 1 + c(s)}\right){\left\| {{\partial _t}u} \right\|_{L_x^2}}.
$$
Therefore it suffices to bound ${\left\| {{h^{ii}}{\partial _i}{A_i}w} \right\|_{{\dot{H}}_x^{ - 1}}}$. By duality and integration by parts, we have
\begin{align}\label{1897}
{\left\| {{h^{ii}}{\partial _i}{A_i}w} \right\|_{{\dot H}_x^{ - 1}}}& = \mathop {\sup }\limits_{{{\left\| g \right\|}_{{\dot H}_x^{  1}}} \le 1} \int_{{{\Bbb  H}^2}} {\left\langle {{h^{ii}}{\partial _i}{A_i}w,g} \right\rangle } {\rm{dvol_h}} \nonumber\\
&\le \mathop {\sup }\limits_{{{\left\| g \right\|}_{{\dot H}_x^{ 1}}} \le 1} \int_{{{\Bbb H}^2}} {\left( {\sqrt {{h^{ii}}} \left| {{A_i}} \right|\left| {\nabla w} \right|\left| g \right| + \left| w \right|\left| {\nabla g} \right|} \right)} {\rm{dvol_h}}.
\end{align}
Then (\ref{XVV4}) follows by (\ref{3.30}), (\ref{XVV3}), (\ref{XVV5}) and (\ref{1897}).
\end{proof}

\begin{lemma}\label{jian}
If $\alpha>0$, then $\phi_s$ in (\ref{heating}) satisfies
\begin{align}\label{axinl}
&{\left\| {{\phi _s}(s)} \right\|_{L_t^\infty L_x^2([T,\infty ) \times {\Bbb H^2})}} \nonumber\\
&\le \left( {{e^{Cs}} + 1 +c(s)} \right)\left( {{{\left\| {{\partial _t}u} \right\|}_{L_t^2L_x^2([T,\infty ) \times {\Bbb H^2})}} + {{\left\| {\nabla {\partial _t}u} \right\|}_{L_t^2L_x^2([T,\infty ) \times {\Bbb H^2})}}} \right)+{\left\| {{\phi _s}(T)} \right\|_{L_x^2}}.
\end{align}
\end{lemma}
\begin{proof}
Expanding $D_i$ in (\ref{heating}) as $\partial_i+\sqrt{-1}A_i$, using Lemma \ref{anxin}, (\ref{uv1}), (\ref{uv5}) we have
\begin{align}
 &{\left\| {{\phi _s}(s)} \right\|_{L_t^\infty L_x^2([T,\infty ) \times {H^2})}}\nonumber\\
 &\lesssim {\left\| {{\phi _s}(T)} \right\|_{L_x^2}} + {\left\| {{\partial _s}w} \right\|_{L_t^2\dot H_x^{ - 1}}} + {\left\| {{h^{ii}}{\partial _i}{A_i}{\phi _s}} \right\|_{L_t^2\dot H_x^{ - 1}}} + {\left\| {{h^{ii}}{A_i}{\partial _i}{\phi _s}} \right\|_{L_t^2L_x^2}} \label{1909}\\
 &+ {\left\| {{h^{ii}}{A_i}{A_i}{\phi _s}} \right\|_{L_t^2L_x^2}} + {\left\| {{h^{ii}}( \phi _s\wedge\phi_i){\phi _i}} \right\|_{L_t^2L_x^2}}+\|A_t\phi_s\|_{L^2_{t}L^2_x}. \label{1910}
\end{align}
The terms in (\ref{1910}), (\ref{1909}) can be estimated by $\|\partial_t u\|_{L^2_{x}}$ as what we have done for terms in (\ref{k908}).  Then the desired bound in (\ref{axinl}) follows from Lemma \ref{hushuo}.
\end{proof}

\subsection{The proof of Theorem 1.1}
Finally, we prove Theorem 1.1 based on Corollary \ref{4.1a} and Lemma \ref{jian}.
\begin{proposition}
Let $u$ be a solution to (\ref{1}) in $C([0,\infty)\times \mathcal{H}^3)$. Then as $t\to \infty$, $u(t,x)$ converges to a harmonic map, namely
$$
\mathop {\lim }\limits_{t \to \infty } \mathop {\lim }\limits_{x\in\Bbb H^2 }{\rm{dist}}_{\Bbb H^2}(u(t,x),Q_{\infty}(x))= 0,
$$
where $Q_{\infty}(x):\Bbb H^2\to \Bbb H^2$ is some harmonic map.
\end{proposition}
\begin{proof}
For $u(t,x)$, by Proposition \ref{3.3}, the corresponding heat flow converges to some harmonic map uniformly for $x\in\Bbb H^2$. Then by the definition of the distance on complete manifolds, we have
\begin{align}\label{ppo0}
{\rm{dis}}{{\rm{t}}_{{\Bbb H^2}}}(u(t,x),Q_{\infty}(x)) \le \int_0^\infty  {{{\left\| {{\partial _s}u} \right\|}_{L_x^\infty }}ds}.
\end{align}
For any $T>0$, $\mu>0$, (\ref{10.112}), (\ref{VI3}) and (\ref{huhu89}) imply
\begin{align}
 \int_T^\infty  {{{\left\| {{\partial _s}u(s,t,x)} \right\|}_{L_x^\infty }}ds}  &\lesssim \int_T^\infty  {{e^{ - \frac{1}{8}s}}{{\left\| {{\partial _t}u(t,x)} \right\|}_{L_x^2}}} ds \lesssim {e^{ - \frac{T}{8}}}{\left\| {{\partial _t}u(t,x)} \right\|_{L_x^2}} \label{mu1}\\
 \int_0^\mu  {{{\left\| {{\partial _s}u(s,t,x)} \right\|}_{L_x^\infty }}ds}  &\lesssim \int_0^\mu  {{{\left\| {{e^{s{\Delta _{{\Bbb H^2}}}}}{\partial _t}u(t,x)} \right\|}_{L_x^\infty }}} ds \le \int_0^\mu  {{s^{ - \frac{1}{2}}}{{\left\| {{\partial _t}u(t,x)} \right\|}_{L_x^2}}} ds \nonumber\\
 &\lesssim {\mu ^{\frac{1}{2}}}{\left\| {{\partial _t}u(t,x)} \right\|_{L_x^2}} \label{mu2}
 \end{align}
Similarly, we have from (\ref{VI3}) that
\begin{align}
 \int_\mu ^T {{{\left\| {{\partial _s}u(s,t,x)} \right\|}_{L_x^\infty }}ds}  &\lesssim \int_\mu ^T {{{\left\| {{e^{(s - \frac{\mu }{2}){\Delta _{{\Bbb H^2}}}}}{\partial _s}u(\frac{\mu}{2},t,x)} \right\|}_{L_x^\infty }}} ds \nonumber\\
 &\lesssim \int_\mu ^T {{{(s - \frac{\mu}{2})}^{ - \frac{1}{2}}}{{\left\| {{\partial _s}u(\frac{\mu}{2},t,x)} \right\|}_{L_x^2}}} ds \nonumber\\
 &\lesssim {\mu ^{ - 1}}\int_\mu ^T {{{\left\| {{\phi _s}(\frac{\mu}{2},t,x)} \right\|}_{L_x^2}}} ds.\label{mu3}
\end{align}
Meanwhile, Lemma \ref{jian} yields for any $T_1>0$, $t\in[T_1,\infty)$
\begin{align}\label{mu4}
&{\left\| {{\phi _s}(\frac{\mu}{2},t,x)} \right\|_{L_x^2}}\nonumber\\
&\lesssim {\left\| {{\phi _s}(\frac{\mu}{2},{T_1},x)} \right\|_{L_x^2}} + C(\mu )\big({\left\| {{\partial _t}u(t,x)} \right\|_{L_t^2L_x^2([{T_1},\infty ) \times {\Bbb H^2}){\rm{ }}}}{\rm{ + }}{\left\| {\nabla {\partial _t}u(t,x)} \right\|_{L_t^2L_x^2([{T_1},\infty ) \times {\Bbb H^2})}}\big).
\end{align}
For any fixed $\epsilon>0$ sufficiently small, choose $T>0$ sufficiently large and $\mu>0$ sufficiently small, then $\|\partial_t u\|_{L^2_x}\le C(u_0)$ proved in Proposition \ref{global} imply
$${e^{ - \frac{T}{8}}}{\left\| {{\partial _t}u(t,x)} \right\|_{L_x^2}}+{\mu ^{\frac{1}{2}}}{\left\| {{\partial _t}u(t,x)} \right\|_{L_x^2}}\le \epsilon^2.$$
Thus (\ref{mu1}) and (\ref{mu2}) are acceptable.
By Corollary \ref{4.1a}, there exists $T_1$ such that
$$C(\mu )\left( {{{\left\| {{\partial _t}u(t,x)} \right\|}_{L_t^2L_x^2([{T_1},\infty ) \times {H^2})}} + {{\left\| {\nabla {\partial _t}u(t,x)} \right\|}_{L_t^2L_x^2([{T_1},\infty ) \times {H^2})}}} \right) \le \mu\epsilon^2.
$$
And since ${\left\| {{\phi _s}(\mu /2,t,x)} \right\|_{L_x^2}}\le \|\partial_t u(t)\|_{L^2}$, we infer from $\int^{\infty}_0\|\partial_t u\|^2_{L^2_x}ds<\infty$ that there exists some $T_2>T_1$ such that
$${\left\| {{\phi _s}(\mu /2,{T_2},x)} \right\|_{L_x^2}}\le \epsilon^2.$$
Then for $t>T_2$, (\ref{mu4}) is acceptable. Finally, we conclude from (\ref{mu3}), (\ref{mu1}), (\ref{mu2}), that for any $\epsilon>0$ sufficiently small, there exists $T_2>0$ such that whenever $t>T_2$, we have
$$\int_0^\infty  {{{\left\| {{\partial _s}u} \right\|}_{L_x^\infty }}ds}  \le \epsilon.
$$
Thus Theorem 1.1 follows by (\ref{ppo0}).
\end{proof}

\section{Appendix}
In the following lemma we collect some important properties of holomorphic harmonic maps between Poincare disks.
\begin{lemma}
Denote $\mathbb{D}=\{z:|z|<1\}$ with the hyperbolic metric to be the Poincare disk. Then any holomorphic map $f:\mathbb{D} \to \mathbb{D}$ is a harmonic map. If we assume that $f(z)$ can be analytically extended into a larger disk than the unit disk, and $f(\mathbb{D})\subset\{z:|z|\le \mu\}$, for some $0<\mu<1$, then the harmonic map $f$ satisfies
\begin{align}
\|df\|_{L^2}&<\infty\label{zui8}\\
e^{r(z)}\|df(z)\|_{L^{\infty}}&<\infty\label{zui9}\\
\|\nabla df\|_{L^2}+\|\nabla^2 df\|_{L^2}&<\infty,\label{zui10}
\end{align}
where $r$ is the geodesic distance between $z$ and the origin in $\mathbb{D}$.
Furthermore, if assume in addition that  for all $z\in\Bbb D$ and some $0<\mu_1\ll1$,
\begin{align}\label{zuiz28}
\sum^{3}_{|\gamma|\le 3}|\partial^\gamma_{x,y}f(z)|\le \mu_1,
\end{align}
then
\begin{align}
\|df\|_{L^2}&\lesssim \mu_1\label{zui81}\\
e^{r(z)}\|df(z)\|_{L^{\infty}}&\lesssim \mu_1\label{zui91}
\end{align}
\end{lemma}
\begin{proof}
Since $f$ can be analytically extended to a larger disk, we can assume all the derivatives of $f(z)$ up to order three are bounded by some constant $C_0$. Recall the metric in the Poincare disk,
\begin{align}
4{\left( {1 - {{\left| z \right|}^2}} \right)^{ - 2}}\left( {dxdx + dydy} \right),
\end{align}
where $z=x+iy$. It is easy to see the Christoffel symbols under the coordinate $(x,y)$ are bounded by $4|z|(1-|z|^2)^{-1}$. Since $|f(z)|\le \mu<1$, for all $z\in\Bbb D$ and some $\mu\in(0,1)$, one has all the Christoffel symbols and metric components on the range $f(\mathbb{D})$ are bounded by some constant say $c(\mu)C_0$, where $c(\mu)$ depends only on $\mu$. Hence the energy density of $f$ is bounded by
\begin{align}\label{zui10}
|df|^2\lesssim C_0c(\mu)(1 - \left| z \right|^2)^2.
\end{align}
Since in the Poincare disk model, $r(z)=\ln(\frac{1+z}{1-z})$, thus we obtain (\ref{zui9}). Using the bound (\ref{zui10}) and the volume ${\rm{dvol}}={\left( {1 -{{\left| z \right|}^2}} \right)^{ - 2}}dxdy$, one immediately deduces (\ref{zui8}). Notice that in the formula of $|\nabla df|$ only Christoffel symbols itself appear (no derivatives of the Christoffel symbols emerge), and the Christoffel symbols on $M=\mathbb{D}$  are bounded by  $4|z|(1-|z|^2)^{-1}$, the Christoffel symbols on $N=\mathbb{D}$ are bounded by $c(\mu)C_0$ explained before.
Hence we still have
\begin{align}\label{zui11}
|\nabla df|^2\lesssim C_0c(\mu)(1 - \left| z \right|^2)^2.
\end{align}
And thus $\|\nabla df\|_{L^2}<\infty$ due to the volume ${\rm{dvol}}={\left( {1 -{{\left| z \right|}^2}} \right)^{ - 2}}dxdy$.
One can easily see by checking the explicit formula of Christoffel symbols under $(x,y)$ that the one order derivatives of them are bounded by
$$8(1-|z|^2)^{-1}+8|z|^2(1-|z|^2)^{-2}.
$$
Thus recalling the Christoffel symbols on $M=\mathbb{D}$  are bounded by  $8(1-|z|^2)^{-2}$ and the Christoffel symbols on $N=\mathbb{D}$ are bounded by $c(\mu)C_0$, we deduce that
\begin{align}\label{zui12}
|\nabla^2 df|^2\lesssim C_0\mu(1 - \left| z \right|^2)^2,
\end{align}
where $C_0$ now relies on the three order derivatives of $f(z)$ on $\mathbb{D}$.
(\ref{zui11}) and (\ref{zui12}) yield (\ref{zui10}).
If we assume in addition that (\ref{zuiz28}) holds, then let $\gamma=0$ in (\ref{zuiz28}), one has $|f(z)|<\frac{1}{2}$. Then all the Christoffel symbols and metric components on the range $f(\mathbb{D})$ are bounded by $C_1$ for some universal constant $C_1>0$. Thus the constant $C_0c(\mu)$ in (\ref{zui10}) to (\ref{zui12}) can be replaced by $C\mu_1$ for some universal constant $C>0$. Hence (\ref{zui81}) and (\ref{zui91}) follow by the same line.
\end{proof}

The following lemma proves the free torsion identity and the commutator identity in (\ref{pknb}), (\ref{commut}). The proof is standard to people working in geometric dispersive equations, we write the details down on one side for reader's convenience and on the other side to emphasize the curved background in our case.
\begin{lemma}
\begin{align}
D_i\phi_j&=D_j\phi_i\label{pknb1}\\
(D_iD_j-D_jD_i)\varphi&\longleftrightarrow{\bf R}(\partial_iu,\partial_ju)(\varphi e).\label{pknb2}
\end{align}
\end{lemma}
\begin{proof}
In fact, by the definitions we have
\begin{align*}
D_i\phi_j&=\partial_i\psi_j^1-[A_j]^2_1\psi_j^2+\sqrt{-1}(\partial_i\psi_j^2+[A_j]^2_1\psi_j^1)\\
&={\partial _i}\left\langle {{\partial _j}u,{e_1}} \right\rangle  - [{A_i}]_1^2\left\langle {{\partial _j}u,J{e_1}} \right\rangle  + \sqrt { - 1} \left( {{\partial _i}\left\langle {{\partial _j}u,J{e_1}} \right\rangle  + [{A_i}]_1^2\left\langle {{\partial _j}u,{e_1}} \right\rangle } \right) \\
&= \left\langle {{\nabla _i}{\partial _j}u,{e_1}} \right\rangle  + \left\langle {{\partial _j}u,{\nabla _i}{e_1}} \right\rangle  - \left\langle
{{\nabla _i}{e_1},J{e_1}} \right\rangle \left\langle {{\partial _j}u,J{e_1}} \right\rangle\\
&+ \sqrt { - 1} \left( {\left\langle {{\nabla _i}{\partial _j}u,J{e_1}} \right\rangle  + \left\langle {{\partial _j}u,{\nabla _i}J{e_1}} \right\rangle  + \left\langle {{\nabla _i}{e_1},J{e_1}} \right\rangle \left\langle {{\partial _j}u,{e_1}} \right\rangle } \right) \\
&= \left\langle {{\nabla _i}{\partial _j}u,{e_1}} \right\rangle  + \sqrt { - 1} \left\langle {{\nabla _i}{\partial _j}u,J{e_1}} \right\rangle,
\end{align*}
which implies (\ref{pknb1}) by the identity ${{\nabla _i}{\partial _j}}u={{\nabla _j}{\partial _i}}u$. Now we turn to (\ref{pknb2}). By definition $D_i=\partial_i+\sqrt{-1}A_i$, we see
\begin{align}\label{1p19bnv}
(D_iD_j-D_jD_i)\varphi=(\partial_iA_j-\partial_jA_i)(\sqrt{-1}\varphi^1-\varphi^2)
\end{align}
It is easy to check $\langle\nabla_je_1,\nabla_ie_2\rangle=0$ by the fact $\{e_1,e_2\}$ is an orthonormal basis of $T\Bbb H^2$. Then we have
\begin{align}\label{2p19bnv}
\partial_iA_j=\langle\nabla_i\nabla_je_1,e_2\rangle+\langle\nabla_je_1,\nabla_ie_2\rangle=\langle\nabla_i\nabla_je_1,e_2\rangle.
\end{align}
By the compatibility of the induced covariant derivative $\nabla$ with the covariant derivative $\widetilde{\nabla}$ in $N=\Bbb H^2$ and the torsion free property, one deduces that
\begin{align}\label{p9b8v}
\nabla_i\nabla_j X-\nabla_j\nabla_i X={\bf R}(\partial_iu,\partial_ju) (X),
\end{align}
for any $X\in u^*TN$.
Meanwhile by the curvature identity,
\begin{align}\label{p19bnv}
\langle{\bf R}(X,Y)(Z),W\rangle=-\langle{\bf R}(Y,X)(Z),W\rangle=-\langle{\bf R}(X,Y)(W),Z\rangle,
\end{align}
(\ref{p9b8v}) further gives
\begin{align}
&\langle\nabla_i\nabla_je_1,e_2\rangle\varphi^1 -\langle\nabla_i\nabla_je_1,e_2\rangle\varphi^1
=\langle {\bf R} (\partial_iu,\partial_ju)(\sum^2_{i=1}\varphi^ie_i),e_2\rangle\label{3p19bnv}\\
&\langle\nabla_i\nabla_je_1,e_2\rangle\varphi^2 -\langle\nabla_i\nabla_je_1,e_2\rangle\varphi^2\nonumber\\
&=-\langle {\bf R} (\partial_iu,\partial_ju)(e_2),e_1\rangle\varphi^2
=-\langle {\bf R} (\partial_iu,\partial_ju)(\sum^2_{i=1}\varphi^ie_i),e_1\rangle.\label{4p19bnv}
\end{align}
Hence (\ref{1p19bnv}), (\ref{2p19bnv}), (\ref{3p19bnv}) and (\ref{4p19bnv}) yield
\begin{align*}
(D_iD_j-D_jD_i)\varphi\longleftrightarrow {\bf R}(\partial_iu,\partial_ju)(\varphi e).
\end{align*}
\end{proof}

\end{document}